\documentclass[journal,onecolumn,draftcls]{IEEEtran}

\usepackage{braket}
\usepackage{array}
\usepackage[caption=false]{subfig}
\usepackage{pgfplots}

\usepackage[hidelinks]{hyperref}
\usepackage{url}            
\usepackage{booktabs}       
\usepackage{amsfonts}       
\usepackage{nicefrac}       
\usepackage{microtype}      
\usepackage{yfonts,dsfont}
\usepackage{setspace}
\usepackage{comment}

\usepackage{subfig}
\usepackage{cite}
\usepackage{pifont}
\usepackage{lipsum}
\usepackage{mathtools}
\usepackage{mathcomp}
\usepackage{euscript,bbm,amstext,wasysym,arydshln,framed}
\usepackage{upgreek}
\usepackage{tikz}
\usepackage{amssymb}
\usetikzlibrary{positioning,shapes,arrows,calc,matrix}
\usepackage{multirow}

\usetikzlibrary{external}

\newcommand{\enma}[1]   {\ensuremath{#1}}

\newcommand{\beq}{\begin{equation}}
\newcommand{\eeq}{\end{equation}}
\newcommand{\bseq}{\begin{subequations}}
\newcommand{\eseq}{\end{subequations}}
\newcommand{\beqn}{\begin{eqnarray}}
\newcommand{\eeqn}{\end{eqnarray}}
\newcommand{\ba}{\begin{array}}
\newcommand{\ea}{\end{array}}
\newcommand{\bct}{\begin{center}}
\newcommand{\ect}{\end{center}}
\newcommand{\btmz}{\begin{itemize}}
\newcommand{\etmz}{\end{itemize}}
\newcommand{\benum}{\begin{enumerate}}
\newcommand{\eenum}{\end{enumerate}}

\newcommand{\norm}[1]{\| #1 \|}                 

\newcommand{\diag}      {\enma{\mathrm{diag}}}

\newcommand{\trace}     {\enma{\mathrm{trace}}}

\newcommand{\matbegin}{
        \left[
}
\newcommand{\matend}{
        \right]
}

\newcommand{\obt}[2]{
  \matbegin \begin{array}{cc}
       #1 & #2
       \end{array} \matend }

\newcommand{\tbt}[4]{
  \matbegin \begin{array}{cc}
       #1 & #2 \\ #3 & #4
       \end{array} \matend }

\newcommand{\be}{\begin{equation}}
\newcommand{\ee}{\end{equation}}

\newcommand{\cplxs}{ C\kern -.35em \rule{0.03 em}{.7 ex}~   }

\def\complex{\hbox{C\kern -.45em \rule{0.03 em}{1.5 ex}}~}

\newcommand{\bi}{\begin{itemize}}
\newcommand{\ei}{\end{itemize}}

\renewcommand{\baselinestretch}{.99} 
\newcommand\mystrut[2]{\vrule width0pt height#1 depth#2\relax}

\newtheorem{mythm}{Theorem}
\newtheorem{myprop}{Proposition}
\newtheorem{mylem}{Lemma}
\newtheorem{myrem}{Remark}

\newtheorem{mycor}{Corollary}

\newcommand{\R}{\mathbb{R}}

\newcommand{\bbR}{\mathbb{R}}

\newcommand{\non}{\nonumber}

\newcommand{\tc}{\textcolor}

\newcommand{\vsp}{\vspace*{0.15cm}}

\newcommand{\Ts}{T_s}

\newcommand{\DefinedAs}[0]{\mathrel{\mathop:}=}
\newcommand{\AsDefined}[0]{=\mathrel{\mathop:}}

\DeclareMathOperator*{\minimize}{minimize}

\DeclareMathOperator{\EX}{\mathbb{E}}

\newcommand{\gd}{\mathrm{gd}}
\newcommand{\na}{\mathrm{na}}
\newcommand{\hb}{\mathrm{hb}}

\newcommand{\Lf}{L}
\newcommand{\mf}{m}

\newcommand{\spec}{\uprho}

\usepackage{tikz}
\usetikzlibrary{3d}
\usetikzlibrary{shapes.symbols,positioning,shapes,arrows,calc,matrix}
\usetikzlibrary{decorations.markings}
\usetikzlibrary{arrows,decorations.pathmorphing}
\usetikzlibrary{positioning,shapes,arrows,calc,matrix}
\usepackage{pgfplots}
\usepgfplotslibrary{fillbetween}
\usepackage{circuitikz}
\usetikzlibrary{patterns,decorations.pathmorphing,quotes,angles}
\usepackage{sidecap}
 
\usepackage{graphicx}

\usepackage{xcolor,colortbl}

\usepackage[english]{babel}
\usepackage{blindtext}

\definecolor{hmaroon}{RGB}{128,0,0}
\definecolor{dred}{rgb}{0.6,0.1,0}

\IEEEoverridecommandlockouts                              
\begin{document}
	
	\title{\LARGE \bf \bf Tradeoffs between convergence rate  
		and  noise amplification \\[0.15 cm]  for momentum-based 
		accelerated optimization algorithms
	}
	
	\author{
		\mbox{Hesameddin Mohammadi, Meisam Razaviyayn, and
			Mihailo R.\ Jovanovi\'c}
		\thanks{Financial support from the National Science Foundation under Awards ECCS~1708906 and ECCS~1809833  is gratefully acknowledged.
		}
		\thanks{Hesameddin Mohammadi and Mihailo R.\ Jovanovi\'c are with the Ming Hsieh Department of Electrical and Computer Engineering, University of Southern California, Los Angeles, CA 90089, USA. Meisam Razaviyayn is with the  Daniel J. Epstein Department of Industrial and Systems Engineering, University of Southern California, Los Angeles, CA 90089, USA (e-mails: hesamedm@usc.edu; mihailo@usc.edu; razaviya@usc.edu). 
		}
	}

	\maketitle
	

	\begin{abstract} 
		We study momentum-based first-order optimization algorithms in which the iterations utilize information from the two previous steps and are subject to an additive white noise. This setup uses noise to account for uncertainty in either gradient evaluation or iteration updates, and it includes Polyak's heavy-ball and Nesterov's accelerated methods as special cases.
		For strongly convex quadratic problems, we use the steady-state variance of the error in the optimization variable to quantify noise amplification and identify fundamental stochastic performance tradeoffs. 
		Our approach utilizes the Jury stability criterion to provide a novel geometric characterization of conditions for linear convergence, and it reveals the relation between the noise amplification and convergence rate as well as their dependence on the condition number and the constant algorithmic parameters.
		This geometric insight leads to simple alternative proofs of standard convergence results and allows us to establish ``uncertainty principle'' of strongly convex optimization: for the two-step momentum method with linear convergence rate, the lower bound on the product between the settling time and noise amplification scales quadratically with the condition number.
		Our analysis also identifies a key difference between the gradient and iterate noise models: while the amplification of gradient noise can be made arbitrarily small by sufficiently decelerating the algorithm, the best achievable variance for the iterate noise model increases linearly with the settling time in the decelerating regime. Finally, we introduce two parameterized families of algorithms that strike a balance between noise amplification and settling time while preserving order-wise Pareto optimality for both noise models. 	\end{abstract}

	
	
	\vspace*{-3ex}
	\section{Introduction}
	
	Accelerated first-order algorithms~\cite{pol64,nes83,nes13} are often used for solving large-scale optimization problems~\cite{botcun05,honrazluopan15} because of their scalability, fast convergence, and low per-iteration complexity. Their convergence properties~\cite{badsei19,susmardahhin13,nes18book,lesrecpac16,hules17,cyrhuvanles18,VanFreLyn18,fazribmor18,schebehol23},  worst-case performance guarantees with respect to function variations~\cite{droteb14,kimfes18,tay17}, and extensions to non-smooth composite optimization problems~\cite{becteb09,ouychelanpas15,baihagzha22} have been carefully studied. However, stochastic performance in the presence of noise has received less attention~\cite{pol77,ben00,macduvada15,yuayinsay16,beirazshatar17,micschebe20}. Prior studies indicate that inaccuracies in the computation of gradient values can adversely impact the convergence rate of accelerated methods and that gradient descent may have advantages relative to its accelerated variants in noisy environments~\cite{luotse93,robmon51,nemjudlansha09,dev-phd13,dvugas16}. In contrast to gradient descent, accelerated algorithms can also exhibit {\em undesirable transient behavior\/}~\cite{doncan15,polsmi19,mohsamjovTAC23}; for convex quadratic problems, the non-normal dynamic modes in accelerated algorithms induce large transient responses of the error in the optimization variable~\cite{mohsamjovTAC23}.
		
	Analyzing the performance of accelerated algorithms with additive white noise that arises from uncertainty in gradient evaluation dates back to Polyak~\cite{pol77}. In this reference, Polyak established the optimal linear convergence rate for strongly convex quadratic problems and used time-varying parameters to obtain convergence in the error variance at a sub-linear rate with an improved constant factor compared to gradient descent.
	Under strong-convexity, noisy algorithms with constant parameters converge at a linear rate to a stationary distribution in Wasserstein distance; the convergence rate along with bounds on transient behavior and steady-state variance were obtained in~\cite{cangurzhu19}. The existence of tradeoffs between the convergence rate and the steady-state variance was also demonstrated.
	
		 Acceleration in a sub-linear regime can also be achieved for smooth strongly convex problems with diminishing stepsize~\cite{dev11} and averaging can be used to prevent the accumulation of gradient noise by accelerated algorithms~\cite{cohdiaore18}. 
	 In~\cite{ghalan13}, it has been further shown that the iteration complexity of any first-order noisy algorithm for strongly convex problems is subject to a fundamental lower bound that consists of bias and variance error terms which decay to zero at linear and sub-linear rates, respectively. To achieve this lower bound, a generic accelerated stochastic approximation framework was developed in~\cite{ghalan13}; this framework can be specialized to obtain optimal or nearly optimal methods. In addition, reference~\cite{aybfalgurozd19b} proposes a multi-stage algorithm based on properly adjusting the stepsize to strike a balance between noise amplification and convergence rate. Therein, the proposed Nesterov-like algorithm with explicit choice of parameters does not require knowledge of noise variance and \mbox{achieves Pareto optimality.}
	 
	 For standard accelerated methods with constant parameters, control-theoretic tools were utilized in~\cite{mohrazjovTAC21} and~\cite{vanles21} to study the steady-state variance of the error in the optimization variable for smooth strongly convex problems. In particular, for the parameters that optimize convergence rates for quadratic problems, tight upper and lower bounds on the noise amplification of gradient descent, heavy-ball method, and Nesterov's accelerated algorithm were developed~\cite{mohrazjovTAC21}. These bounds are expressed in terms of the condition number $\kappa$ and the problem dimension $n$, and they demonstrate opposite trends relative to the settling time: {\em for a fixed problem dimension $n$, accelerated algorithms increase noise amplification by a factor of $\Theta(\sqrt{\kappa})$ relative to gradient descent.\/} Similar result also holds for heavy-ball and Nesterov's algorithms with parameters that provide convergence rate $\rho \leq 1 - c/\sqrt{\kappa}$ with $c > 0$~\cite{mohrazjovTAC21}. Furthermore, for strongly convex optimization problems with a condition number $\kappa$, tight and attainable upper bounds for noise amplification of gradient descent and Nesterov's accelerated method were provided~\cite{mohrazjovTAC21}. 
		
	In this paper, we extend the results of~\cite{mohrazjovTAC21} to the class of first-order algorithms with three constant parameters in which the iterations involve information from the two previous steps. This class includes heavy-ball and Nesterov's accelerated algorithms as special cases and we examine its stochastic performance for strongly convex quadratic problems. Our results are complementary to~\cite{aybfalgurozd20}, which evaluates stochastic performance in both the objective and the iterate errors, and to a recent work~\cite{vanles21} which studies the steady-state variance of the error associated with the point at which the gradient is evaluated. Reference~\cite{vanles21} combined theory with computational experiments to demonstrate that a parameterized family of heavy-ball-like methods with reduced stepsize provides Pareto-optimal algorithms for simultaneous optimization of convergence rate and amplification of gradient noise. In contrast to~\cite{vanles21}, we establish analytical lower bounds on the product of the settling time and the steady-state variance of the error in the optimization variable that hold for any constant stabilizing parameters and for both gradient and iterate noise models. Our lower bounds scale with the square of the condition number and reveal a fundamental limitation of this class of algorithms. 
	
	In another related work~\cite{cangur22}, the tradeoff between the convergence rate and risk of sub-optimality for the class of two-step momentum algorithms was studied. Reference~\cite{cangur22} characterized the convergence rate and steady-state variance and proposed a systematic and computationally tractable approach based on solutions of semidefinite programming problems to achieve tradeoffs in the risk-averse setting for strongly convex problems. The impact of parameters on the convergence rate and steady-state variance for momentum-based algorithms with extensions to non-convex problems has also been studied~\cite{gitlanzhaxia19}. Therein, the authors utilized second-order Taylor series expansion in the stepsize to reveal non-intuitive trends for the effect of momentum parameters on the stationary variance. In the context of differential privacy, theoretical bounds along with numerical observations were used to quantify noise amplification of noisy accelerated algorithms~\cite{kurilkguryil22}. In addition, reference~\cite{kurilkguryil22} obtained optimal noise-robust heavy-ball algorithm and proposed multi-stage variants of accelerated algorithms that attenuate noise in the gradient while enjoying an improved convergence rate.
	
	In addition to considering noise in gradient evaluation, we study the stochastic performance of algorithms when noise is directly added to the iterates (rather than the gradient). For this iterate noise model, we establish an alternative lower bound on the noise amplification. This bound scales linearly with the settling time and is order-wise tight for settling times that are larger than that of gradient descent with the standard stepsize. In this decelerated regime, we identify a key difference between the two noise models: while the impact of gradient uncertainties on variance amplification can be made arbitrarily small by decelerating the two-step momentum algorithm, the best achievable variance for the iterate noise model increases linearly with the settling time in the decelerated regime.
	
	Our results build upon a simple, yet powerful geometric viewpoint, which clarifies the relation between condition number, convergence rate, and algorithmic parameters for strongly convex quadratic problems. This viewpoint allows us to present alternative proofs for (i) the optimal convergence rate of the two-step momentum algorithm, which recovers Nesterov's fundamental lower bound on the convergence rate~\cite{nes04book} for {\em finite dimensional problems\/}~\cite{arjshasha16}; and (ii) the optimal rates achieved by standard gradient descent, heavy-ball method, and Nesterov's accelerated algorithm~\cite{lesrecpac16}. In addition, it enables a novel geometric characterization of noise amplification in terms of stability margins and it allows us to precisely quantify tradeoffs between convergence rate and robustness to noise. 
		
	We also introduce two parameterized families of algorithms that are structurally similar to the heavy-ball and Nesterov's accelerated algorithms. These algorithms utilize continuous transformations from gradient descent to the corresponding accelerated algorithm (with the optimal convergence rate) via a homotopy path, and they can be used to provide additional insight into the tradeoff between convergence rate and noise amplification. We prove that these parameterized families are order-wise (in terms of the condition number) Pareto-optimal for simultaneous minimization of settling time and noise amplification. Another family of algorithms that facilitates similar tradeoff was proposed in~\cite{cyrhuvanles18}, and it includes the fastest known algorithm for the class of smooth strongly convex problems. We also utilize negative momentum parameters to decelerate a heavy-ball-like family of algorithms relative to gradient descent with the optimal stepsize. For both noise models, our parameterized family yields order-wise optimal algorithms and it allows us to further highlight the key difference between them in the decelerated regime.
	
	 In contrast to conjugate gradient methods that are exceedingly sensitive to poor-conditioning and noise, momentum-based first-order optimization algorithms are flexible enough to (i) be deployed to complex optimization landscapes and environments; and to (ii) benefit from recent extensive hardware optimization and parallelization in modern platforms that utilize GPUs. In spite of broad applicability of these algorithms with constant parameters, a clear understanding of fundamental tradeoffs between variance amplification and convergence rate is not available in the existing literature.  Our paper addresses this issue by (i) providing a novel geometric insight into the dependence of convergence rate and variance amplification on the algorithmic parameters; (ii) identifying the product between the variance amplification $J$ and the settling time $T_s$, $J\times T_s$, as an important ``conserved quantity'' of two-step momentum algorithms; and (iii) establishing {\em tight bounds\/} on $J\times T_s$ in terms of the \tc{blue}{square of the} condition number.
	
	The rest of the paper is organized as follows. In Section~\ref{sec.PerilimBack}, we provide preliminaries and background material and, in Section~\ref{sec.mainResult}, we summarize our key contributions. In Section~\ref{sec.modalDecomp}, we introduce the tools and ideas that enable our analysis. In particular, we utilize the Jury stability criterion to provide a novel geometric characterization of stability and $\rho$-linear convergence and exploit this insight to derive simple alternative proofs of standard convergence results and quantify fundamental stochastic performance tradeoffs. In Section~\ref{sec.Tuning}, we introduce two parameterized families of algorithms that allow us to constructively tradeoff settling time and noise amplification. In Section~\ref{sec.mainResultProof}, we provide proofs of our main results and, in Section~\ref{sec.concludingRemarks}, we conclude the paper. 						   
		
	\vspace*{-1ex}
	\section{Preliminaries and background}
	\label{sec.PerilimBack}
	
	For the unconstrained optimization problem
	\begin{align}\label{eq.f}
		\minimize\limits_{x} ~f(x)
	\end{align} 
	where  $f$: $\R^n\rightarrow\R$ is a strongly convex function with a Lipschitz continuous gradient $\nabla f$, we consider noisy momentum-based first-order algorithms that use information from the two previous steps to update the optimization variable:
	\be
	\ba{rcl}
	x^{t+2} 
	& \!\!\! = \!\!\! &
	x^{t+1}
	\; + \;
	\beta
	(
	x^{t+1}
	\, - \,  
	x^{t}
	)
	\, - \, 
	\alpha 
	\nabla 
	f
	\!
	\left(
	x^{t+1}
	\, + \,
	\gamma
	(
	x^{t+1}
	\, - \,  
	x^{t})
	\right)	
	\, + \;
	\sigma_w w^t.
	\ea
	\label{alg.TM_disc}
	\ee
	Here, $t$ is the iteration index, $\alpha$ is the stepsize, $\beta$ and $\gamma$ are momentum parameters, $\sigma_w$ is the noise magnitude, and $w^t$ is an additive white noise with zero mean and identity covariance,
	\[
	\EX 
	\left[
	w^t
	\right]
	\, = \; 
	0,
	~
	\EX 
	\left[
	w^t
	(w^{\tau})^T
	\right]
	\, = \; 
	I\,\delta (t \, - \, \tau)
	\]
	where $\delta$ is the Kronecker delta and $\EX$ is the expected value operator. In this paper, we consider two noise models.
	\begin{enumerate}
		\item Iterate noise ($\sigma_w=\sigma$): models uncertainty in computing the iterates of~\eqref{alg.TM_disc}, where $\sigma$ denotes the stepsize-independent noise magnitude.
		
		\item Gradient noise ($\sigma_w=\alpha \sigma$): models uncertainty in the gradient evaluation. In this case, the stepsize $\alpha$ directly impacts magnitude of the additive noise. 		
	\end{enumerate}
	
	Iterate noise model captures scenarios with uncertainties in optimization variables because of roundoff, quantization, and communication errors. This model has also been used to improve generalization and robustness in machine learning~\cite{he2019parametric}.  On the other hand, the second noise model accounts for gradient computation error or scenarios in which the gradient is estimated from noisy measurements~\cite{pol87}. Also, noise may be intentionally added to the gradient for privacy reasons~\cite{bassily2014private}. 
	
	\vsp 
	\begin{myrem}
		An alternative noise model with $\sigma_w = \sqrt{\alpha} \sigma$ has been used to escape local minima in stochastic gradient descent~\cite{gelmit91} and to provide non-asymptotic guarantees in non-convex learning~\cite{ragraktel17,zhaliacha17}. This model arises from a discretization of the continuous-time Langevin diffusion dynamics~\cite{ragraktel17} and, for strongly convex quadratic problems, our framework can be used to examine acceleration/robustness tradeoffs. For algorithms that are faster than the standard gradient descent, this model has order-wise identical performance bounds as the other two models and the only difference arises in decelerated regime. We omit details for brevity.
	\end{myrem}
	\vsp
	
	Special cases of~\eqref{alg.TM_disc} include noisy gradient descent ($\beta=\gamma=0$), Polyak's heavy-ball method $(\gamma=0)$, and Nesterov's accelerated algorithm $(\gamma=\beta)$. In the absence of noise (i.e., for $\sigma=0$), the parameters $(\alpha,\beta,\gamma)$ can be selected such that the iterates converge linearly to the globally optimal solution~\cite{nes18book}. For the family of smooth strongly convex problems, the parameters that yield the fastest known linear convergence rate were provided in~\cite{VanFreLyn18}.
		
		\vspace*{-1ex}
	\subsection{Linear dynamics for quadratic problems}
	\label{subsec.Quadratic}
	
	Let $\mathcal{Q}_{\mf}^{\Lf}$ denote the class of $\mf$-strongly convex $\Lf$-smooth quadratic functions 
	\begin{align}
		f(x)  
		\; = \;
		\tfrac{1}{2}
		\,
		x^T Q \, x
		\; - \; 
		q^T x
		\label{eq.quadraticObjective}
	\end{align}
	with the condition number $\kappa \DefinedAs \Lf/\mf$, where $q$ is a vector and $Q = Q^T \succ 0$ is the Hessian matrix with eigenvalues
	\[
	\Lf 
	\; = \;
	\lambda_1 
	\; \geq \; 
	\lambda_2 
	\; \geq \; 
	\ldots 
	\; \geq \;
	\lambda_n 
	\; = \; 
	\mf
	\; > \; 
	0.
	\]
	For the quadratic objective function in~\eqref{eq.quadraticObjective}, we can use a linear time-invariant (LTI) state-space model to describe the {\em two-step momentum algorithm\/}~\eqref{alg.TM_disc} with constant parameters,
	\begin{subequations}
		\be
		\ba{rcl}
		\psi^{t+1} 
		& \!\!\! = \!\!\! &
		A \,\psi^{t} 
		\; + \; 
		B \, w^t
		\\[0.1cm]
		z^t
		& \!\!\! = \!\!\! &
		C \, \psi^t
		\label{eq.ss}
		\ea
		\ee
		where $\psi^t$ is the state, $z^t \DefinedAs x^t-x^\star$ is the performance output that determines the error to the optimal solution $x^\star$, and $w^t$ is the white stochastic input. In particular, choosing
		$
		\psi^t 
		\DefinedAs 
		[  (x^t-x^\star)^T \; (x^{t+1}-x^\star)^T  ]^T
		$
		yields
		\be
		\hspace{-0.1 cm}
		\ba{rcl}
		A
		&\!\!\! = \!\!\!&
		\tbt{0}{I}{-\beta I +\gamma\alpha Q}{(1+\beta)I - (1+\gamma)\alpha Q},
		~~
		B^T
		\; = \;
		\obt{0}{\sigma_w I},
		~~
		C
		\; = \; 
		\obt{I}{0}.
		\ea
		\ee
		\label{eq.two-stepMomentumSS}
	\end{subequations}
	
	\vspace*{-0.4 cm}
	\subsection{Convergence rate}
	
	An algorithm is stable if in the absence of noise (i.e., $\sigma_w=0$),
	the state converges linearly with some rate $\rho<1$, 
	\begin{equation}\label{eq.linearConvergenceFormal}
		\norm{\psi^t}_2
		\; \le \; 
		c_t\, \rho^t \, \norm{\psi^0}_2~~\text{for all } t\;\ge\;1 
	\end{equation}
	for all $f\in \mathcal{Q}_{\mf}^{\Lf}$ and all initial conditions $\psi^0$, where $c_t>0$ grows at most polynomially with $t$. For LTI system~\eqref{eq.ss}, the spectral radius $\spec(A)$ determines the best achievable convergence rate. In addition,
	\begin{align}\label{eq.SettlingTimeDef}
		\Ts \;\DefinedAs\; {1}/({1\,-\,\rho})
	\end{align}
	provides the first-order approximation in $\epsilon$ to the {\em settling time\/}, i.e., the number of iterations required to reach a given desired accuracy $\epsilon$; see Appendix~\ref{app.settlingTime}.
	For the class $\mathcal{Q}_{\mf}^{\Lf}$ of high-dimensional functions (i.e., for $n\gtrsim\Ts$), Nesterov established the fundamental lower bound on the settling time (convergence rate) of any first-order algorithm~\cite{nes18book}, 
	\begin{align}\label{eq.NesFunLow}
		\Ts
		\; \ge \;
		(\sqrt{\kappa}\,+\,1)/{2}.
	\end{align}
	This lower bound is sharp and it is achieved by the heavy-ball method with the parameters provided in Table~\ref{tab:rates}~\cite{lesrecpac16}.
	We note that polynomial factors $c_t$ may appear because of non-monotonic transient responses for non-normal dynamics~\cite{mohsamjovTAC23,cangur22}. In addition, the restriction $n\gtrsim\Ts$ in~\eqref{eq.NesFunLow} can be lifted for the class of two-step momentum algorithms with constant parameters~\cite{arjshasha16}. Robust control techniques have also been used to extend this result to algorithms that involve more than two previous steps~\cite{ugrpetsha23}.

	\begin{table*}
		\centering
		\begin{tabular}{ |{l}|{l}|{l}|{l}| {l}|} 
			\hline \mystrut{0.35 cm}{0.1 cm}
			method &  fastest parameters $(\alpha,\beta,\gamma)$ &  $\Ts$ &
			$J_{\min}/\sigma_w^2$ & $J_{\max}/\sigma_w^2$
			\\
			\hline
			\mystrut{0.4 cm}{0.1 cm}
			Gradient
			& 
			$(2/({\Lf+\mf}),\,0,\,0)$
			& 
			$(\kappa +1)/2$ 
			&
			$
			\Theta(\kappa) + n
			$
			&
			$ n \Theta(\kappa)$
			\\[0.10cm] 
			Heavy-ball
			& 
			$(4/(\sqrt{\Lf}+\sqrt{\mf})^2,\,( 1-2/( \sqrt{\kappa} + 1 ))^2,\,0)$
			& 
			$(\sqrt{\kappa} + 1 )/2$ 
			&
			$
			\Theta(\kappa \sqrt{\kappa}) + n \Theta (\sqrt{\kappa})
			$
			&
			$
			n \Theta(\kappa \sqrt{\kappa})
			$
			\\[0.15cm] 
			Nesterov
			& 
			$({4}/({3\Lf+\mf}),\,1-4/({\sqrt{3\kappa+1}+2},\,\beta)$
			& 
			$\sqrt{3\kappa+1}/2  \!\!\!$ 
			&
			$ 
			\Theta(\kappa \sqrt{\kappa}) +n
			$
			&
			$
			n \Theta(\kappa \sqrt{\kappa})
			$
			\\[-.35cm]
			& 
			& 
			&
			&
			\\
			\hline
		\end{tabular}		
		\caption{Settling times $\Ts\DefinedAs1/(1 - \rho)$~\cite[Proposition~1]{lesrecpac16} along with the corresponding noise amplification bounds in~\eqref{eq.JmaxJminDef}~\cite[Theorem~4]{mohrazjovTAC21} for the parameters that optimize the linear convergence rate $\rho$ for strongly convex quadratic function $f\in\mathcal{Q}_{\mf}^{\Lf}$ with the condition number $\kappa\DefinedAs \Lf / \mf$. Here, $n$ is the dimension of $x$ and $\sigma_w^2$ is the variance of the white noise.}
		\vspace*{-0.8 cm}
		\label{tab:rates}
	\end{table*}

	\vspace*{-0.325cm}
	\subsection{Noise amplification}
	
	For LTI system~\eqref{eq.ss} driven by an additive white noise $w^t$,
	$
	\EX \left( \psi^{t+1} \right) 
	=  
	A  \EX \left( \psi^t \right).
	$ 
	Thus, $\EX \left( \psi^t \right) = A^t  \EX \left( \psi^0 \right)$ and, for any stabilizing parameters $(\alpha,\beta,\gamma)$, the iterates reach a statistical steady-state with $\lim_{t \, \to \, \infty} \EX \left( \psi^t \right) = 0$ and a variance that can be computed from the solution of the algebraic Lyapunov equation~\cite{kwasiv72,mohrazjovTAC21}. We call the steady-state variance of the error in the optimization variable noise (or variance) amplification,
	\be
	J
	\; \DefinedAs \;
	\lim_{t \, \to \, \infty} 
	\, 
	\dfrac{1}{t}
	\sum_{k \, = \, 0}^{t}
	\EX
	\left( \norm{x^k \, - \, x^\star}_2^2 \right).
	\label{eq.Jnew}
	\ee
	In addition to the algorithmic parameters $(\alpha,\beta,\gamma)$, the entire spectrum $\{\lambda_i\,|\, i = 1,\ldots, n\}$ of the Hessian matrix $Q$ impacts the noise amplification $J$ of algorithm~\eqref{alg.TM_disc}~\cite{mohrazjovTAC21}. 
		\begin{myrem}\label{rem.diffPerformanceMetric}
		An alternative performance metric that examines the steady-state variance of $y^t - x^\star$ was considered in~\cite{vanles21}, where $y^t \DefinedAs x^t + \gamma(x^t-x^{t-1}) $ is the point at which the gradient is evaluated in~\eqref{alg.TM_disc}. For all $\gamma\ge0$, we have
		$
		J_{x}
		\le
		J_{y}
		\le
		(1+2\gamma)^2
		J_{x}
		$,
		where the subscripts $x$ and $y$ denote the noise amplification in terms of the error in $x^t$ and $y^t$. Thus, these performance metrics are within a constant factor of each other for bounded values of non-negative momentum parameter $\gamma$. 
	\end{myrem}
	
	\vspace*{-2ex}
	\subsection{Parameters that optimize convergence rate}
	
For special instances of the two-step momentum algorithm~\eqref{alg.TM_disc} applied to strongly convex quadratic problems, namely gradient descent ($\gd$), heavy-ball method ($\hb$), and Nesterov's accelerated algorithm ($\na$), the parameters that yield the fastest convergence rates were established in~\cite{pol77,lesrecpac16}. These parameters along with the corresponding rates and the noise amplification bounds are provided in Table~\ref{tab:rates}. The convergence rates are determined by the spectral radius of the corresponding $A$-matrices and the noise amplification bounds are computed by examining the solution to the algebraic Lyapunov equation and determining the functions $f \in \mathcal{Q}_{\mf}^{\Lf}$ for which the steady-state variance is maximized/minimized~\cite[Proposition~1]{mohrazjovTAC21}. Since the optimal rate for the heavy-ball method meets the fundamental lower bound~\eqref{eq.NesFunLow}, this choice of parameters also optimizes the convergence rate of~\eqref{alg.TM_disc} for $f \in \mathcal{Q}_{\mf}^{\Lf}$.

For the optimal parameters provided in Table~\ref{tab:rates}, there is a $\Theta(\sqrt{\kappa})$ improvement in settling times of the heavy-ball and Nesterov's accelerated algorithms relative to gradient descent, 
\beq
\label{eq.settlingTimeOptimalRates}
\Ts \; = \,
\left\{
\ba{rl}
\Theta(\kappa) & \gd
\\[0.10cm]
\Theta(\sqrt{\kappa}) & \hb, \na
\ea
\right.
\eeq
where $a=\Theta(b)$ means that $a$ lies within constant factors of $b$ as $b\rightarrow\infty$. This improvement makes accelerated algorithms popular for problems with large condition number $\kappa$. 

While convergence rate is only affected by the largest and smallest eigenvalues of $Q$, the entire spectrum of $Q$ influences the noise amplification $J$. On the other hand, the largest and smallest values of $J$ over the function class $\mathcal{Q}_{\mf}^{\Lf}$,
\beq\label{eq.JmaxJminDef}
J_{\max} 
\, \DefinedAs \,
\displaystyle{\max_{f \, \in \, \mathcal{Q}_{\mf}^{\Lf}} J}
,\qquad
J_{\min} 
\, \DefinedAs \,
\displaystyle{\min_{f \, \in \, \mathcal{Q}_{\mf}^{\Lf}} J}
\eeq 
depend only on the noise magnitude $\sigma_w$, the algorithmic parameters ($\alpha,\beta,\gamma$), the problem dimension $n$, and the extreme eigenvalues $\mf$ and $\Lf$ of $Q$. 

For the parameters that optimize convergence rates, tight upper and lower bounds on the noise amplification were developed in~\cite[Theorem~4]{mohrazjovTAC21}. These bounds are expressed in terms of the condition number $\kappa$ and the problem dimension $n$, and they demonstrate opposite trends relative to the settling time. In particular, for gradient descent,
\begin{subequations}
	\label{eq.noiseamplificationOptimalRates}
	\beq \label{eq.NoiseAmpOptimalRatesGradient}
	J_{\max} 
	\; = \;
	\sigma_w^2 n \Theta(\kappa)
	,\qquad
	J_{\min} 
	\; = \;
	\sigma_w^2 
	(\Theta(\kappa) + n)
	\eeq
	and for accelerated algorithms,
	\be\label{eq.NoiseAmpOptimalRatesAccelerated}
	\ba{rcl}
	J_{\max} 
	&\!\!\!=\!\!\!&
	\sigma_w^2 n \Theta(\kappa \sqrt{\kappa})
	\\[0.18 cm]
	J_{\min} 
	&\!\!\!=\!\!\!&
	\left\{
	\ba{rl}
	\sigma_w^2 
	(\Theta(\kappa \sqrt{\kappa}) \,+\, n \Theta (\sqrt{\kappa})) & \hb
	\\[0.15cm]
	\sigma_w^2 
	(\Theta(\kappa \sqrt{\kappa}) \,+\, n) & \na.
	\ea
	\right.
	\ea
	\ee
\end{subequations}

Thus, for fixed problem dimension $n$ and noise magnitude $\sigma_w$, accelerated algorithms increase noise amplification by a factor of $\Theta(\sqrt{\kappa})$ relative to gradient descent for the parameters that optimize convergence rates. While similar result also holds for heavy-ball and Nesterov's algorithms with arbitrary values of parameters $\alpha$ and $\beta$ that provide settling time $\Ts \leq c\sqrt{\kappa}$ with $c > 0$~\cite[Theorem~8]{mohrazjovTAC21}, in this paper we establish fundamental tradeoffs between noise amplification and settling time for the class of the two-step momentum algorithms~\eqref{alg.TM_disc} with arbitrary stabilizing values of constant parameters ($\alpha,\beta,\gamma$).

\vspace*{-2ex}
	\section{Summary of main results}
	\label{sec.mainResult}
	
	In this section, we summarize our key contributions regarding robustness/convergence tradeoff for noisy two-step momentum algorithm~\eqref{alg.TM_disc}. In addition, our geometric characterization of stability and $\rho$-linear convergence allows us to provide alternative proofs of standard convergence results and quantify fundamental performance tradeoffs. The proofs of results presented here can be found in Section~\ref{sec.mainResultProof}.  
			
	\vspace*{-2ex}
	\subsection{Bounded noise amplification for stabilizing parameters} 
	
	For a discrete-time LTI system with a convergence rate $\rho$, the distance of the eigenvalues to the unit circle is larger than $1-\rho$. We use this stability margin to establish an upper bound on the noise amplification $J$ of the two-step momentum method~\eqref{alg.TM_disc} for {\em any\/} stabilizing parameters $(\alpha,\beta,\gamma)$. 
	
	\vsp
	\begin{mythm}[Upper bounds]
	\label{thm.mainUpperbound}
		Let the parameters $(\alpha,\beta,\gamma)$ be such that the two-step momentum algorithm~\eqref{alg.TM_disc} converges linearly with the rate $\rho=1-1/\Ts$ for all $f\in\mathcal{Q}_{\mf}^{\Lf}$. Then, 
		\begin{subequations}\label{eq.mainUpperBound}
			\begin{align}\label{eq.mainUpperBoundsigma1}
				J
				\;\le\;
				\dfrac{\sigma_w^2 (1\,+\,\rho^2)}{(1\,+\,\rho)^3} \, n \, \Ts^3
			\end{align}
			where $n$ is the problem size. 		
			Furthermore, for the gradient noise model ($\sigma_w=\alpha\sigma$), 	
			\begin{align}\label{eq.mainUpperBoundsigmaalpha}
				J
				\;\le\;
				\dfrac{\sigma^2 (1\,+\,\rho)(1\,+\,\rho^2)}{\Lf^2 } \, n \, \Ts^3.
			\end{align}
		\end{subequations}
	\end{mythm}
	
	For $\rho<1$, both upper bounds in~\eqref{eq.mainUpperBound} scale with $n \Ts^3$ and they are exact for the heavy-ball method with the parameters that optimize the convergence rate provided by Table~\ref{tab:rates}. However, these bounds are not tight for all stabilizing parameters; e.g., applying~\eqref{eq.mainUpperBoundsigma1} to gradient descent with the optimal stepsize $\alpha = 2/(\Lf+\mf)$ yields
	$
	J
	\le
	\sigma_w^2 n \Theta(\kappa^3),
	$
	which is off by a factor of $\kappa^2$; cf.~Table~\ref{tab:rates}. The bound in~\eqref{eq.mainUpperBoundsigmaalpha} is obtained by combining~\eqref{eq.mainUpperBoundsigma1} with 
	$
	\alpha L \le  (1 + \rho)^2,
	$
	which follows from  the conditions for $\rho$-linear convergence in Section~\ref{sec.modalDecomp}.  For entropic risk, bounds on noise amplification were derived in~\cite{cangur22}. In contrast, Theorem~\ref{thm.mainUpperbound} uses a geometric viewpoint to capture an explicit, exact cubic dependence of $J_{\max}$ on the settling time.
		
	\vspace*{-2ex}
	\subsection{Tradeoff between settling time and noise amplification}
	
	For any stabilizing constant parameters ($\alpha,\beta,\gamma$) in the two-step momentum algorithm~\eqref{alg.TM_disc}, we next establish lower bounds on the largest and the smallest noise amplification $J_{\max}$ and $J_{\min}$ over the class of functions $\mathcal{Q}_{\mf}^{\Lf}$, as defined in~\eqref{eq.JmaxJminDef}, in terms of the settling time $\Ts$. 
		\vsp

	\begin{mythm}[Reciprocal lower bounds]\label{thm.LowerJ} Let the parameters $(\alpha,\beta,\gamma)$ be such that the two-step momentum algorithm~\eqref{alg.TM_disc} converges linearly with the rate $\rho=1-1/\Ts$ for all $f\in\mathcal{Q}_{\mf}^{\Lf}$. Then, $J_{\max}$ and $J_{\min}$ in~\eqref{eq.JmaxJminDef} satisfy,
		\begin{subequations}
			\begin{align}
				J_{\max}
				&\;\ge\;
				\sigma_w^2
				\!
				\left((n\,-\,1) \dfrac{\kappa^2}{64}
				\,+\,
				\dfrac{\sqrt{\kappa}+1}{2}\right) \Ts^{-1}
				\label{eq.thmLowerJMax}
				\\[-0.cm]
				J_{\min}
				&\;\ge\;
				\sigma_w^2
				\!
				\left(\dfrac{\kappa^2}{64}
				\,+\,
				(n\,-\,1) 
				\dfrac{\sqrt{\kappa}+1}{2}\right) \Ts^{-1}.
				\label{eq.thmLowerJMin}
			\end{align}
				\end{subequations}
			Furthermore, for the gradient noise model ($\sigma_w=\alpha\sigma$), we have
				\begin{subequations}
			\begin{align}	
				\hspace{-0.1 cm}
				J_{\max}
				&\;\ge\;
				\dfrac{\sigma^2}{\Lf^2}
				\!
				\left(
				(n\,-\,1) 
				\dfrac{\kappa^2}{4}
				\,+\,
				\max\left\{ \dfrac{\kappa^2}{\Ts^3}, \, \dfrac{1}{4}\right\}
				\right)\Ts^{-1}
				\label{eq.thmLowerJMax_withAlpha}
				\\[-0.cm]
				\hspace{-0.1 cm}
				J_{\min}
				&\;\ge\;
				\dfrac{\sigma^2}{\Lf^2}
				\!
				\left(
				\dfrac{\kappa^2}{4}
				\,+\,
				(n\,-\,1) 
				\max\left\{ \dfrac{\kappa^2}{\Ts^3}, \, \dfrac{1}{4}\right\}
				\right)\Ts^{-1}.
				\label{eq.thmLowerJMin_withAlpha}
			\end{align}
		\end{subequations}
	\end{mythm}
	
	For both noise models, the condition number $\kappa$ restricts the performance of the two-step momentum algorithm with constant parameters: {\em for a fixed problem size $n$, all four lower bounds in Theorem~\ref{thm.LowerJ} demonstrate the quadratic dependence on the condition number $\kappa$ for both $J_{\max}\times \Ts$ and $J_{\min}\times \Ts$.\/} Relative to the dominant term in $\kappa$, the problem dimension $n$ appears in a multiplicative fashion for the lower bounds on $J_{\max}$ and in an additive fashion for the lower bounds on $J_{\min}$.  We note that the fundamental lower bound on $\Ts$ in~\eqref{eq.NesFunLow} holds for large problem dimension $n$ for any first-order algorithm. In contrast, Theorem~\ref{thm.LowerJ} holds for arbitrary $n$ for the class of two-step momentum algorithms with constant parameters.
	
	\vsp
	
		\begin{mythm}[Linear lower bounds]\label{thm.LowerBoundSettlingTimeBased}
		Let the parameters $(\alpha,\beta,\gamma)$ be such that the two-step momentum algorithm~\eqref{alg.TM_disc}
		achieves the convergence rate $\rho=\spec(A)=1-1/\Ts$, where the matrix $A$ is given by~\eqref{eq.two-stepMomentumSS}. Then, $J_{\max}$ and $J_{\min}$ satisfy,
		\begin{subequations}\label{eq.LowerBoundSettlingTimeBased}
			\begin{align}
				J_{\max}
				&\;\ge\;
				\sigma_w^2\left((n \, - \, 1) \, \dfrac{\Ts}{2(1\,+\,\rho)^2}  \,+\, 1\right)
				\label{eq.JmaxLowerBoundSettlingTimeBased}
				\\[-0.cm]
				J_{\min}
				&\;\ge\;
				\sigma_w^2\left(\dfrac{ \Ts}{2(1\,+\,\rho)^2}  \,+\, (n \, - \, 1) \right).
				\label{eq.JminLowerBoundSettlingTimeBased}
			\end{align}
		\end{subequations}
	\end{mythm}
	
	We  observe that the lower bounds on $J_{\max}$ and $J_{\min}$ in Theorem~\ref{thm.LowerBoundSettlingTimeBased} grow linearly with $\Ts$.

	\subsection{Accuracy of lower bounds}

	In this subsection, we  establish upper bounds on $J_{\max}$ and $J_{\min}$ for a parameterized family of heavy-ball-like algorithms in terms of the settling time $\Ts$. By comparison to the lower bounds established in Theorems~\ref{thm.LowerJ} and~\ref{thm.LowerBoundSettlingTimeBased}, we prove that for any settling time $\Ts$ these bounds are {\em order-wise tight\/} in $\kappa$.
	\vsp
	\begin{mythm}[Upper bounds]\label{thm.mainResultHBLike}
		For the class of strongly convex quadratic functions $\mathcal{Q}_{\mf}^{\Lf}$ with the condition number $\kappa=\Lf/\mf$, let the scalar $\rho$ be such that the fundamental lower bound $\Ts = 1/(1 - \rho)\ge (\sqrt{\kappa} + 1)/2$ given by~\eqref{eq.NesFunLow} holds.
		Then, the two-step momentum algorithm~\eqref{alg.TM_disc} with parameters
		\beq
		\alpha
		\,=\,
		\dfrac{(1 + \rho)(1 + \beta/\rho)}{\Lf},
		~~
		\beta
		\,=\,
		\rho 
		\,
		\dfrac{\kappa - (1 + \rho)/(1 - \rho)}{\kappa + (1 + \rho)/(1 - \rho)}
		\label{eq.HBparameters}
		\eeq
and $\gamma=0$, converges linearly with the rate $\rho$. In addition,  $J_{\max}$ and $J_{\min}$ in~\eqref{eq.JmaxJminDef} satisfy
\begin{subequations}\label{eq.thm.mainResultHBLikeLinearGrowthT}
	\begin{align}\label{eq.thm.mainResultHBLikeJmaxLinearGrowth}
		J_{\max}
		&\,\le\,
		\left\{
		\ba{ll}
			\sigma_w^2 \, n \kappa \, \left((\kappa \,+\, 1) /2\right) \Ts^{-1} & \text{if}\quad\Ts\le \tau
			\\[0.15 cm]
			\sigma_w^2n\,\Ts & \text{if}\quad \Ts\ge  \tau
		\ea
		\right.\!\!
		\\[0.15 cm]
		\label{eq.thm.mainResultHBLikeJminLinearGrowth}
		J_{\min}
		&\,\le\,
		\left\{
		\ba{ll}
			 \sigma_w^2 \, \kappa \, \left(
			 \kappa \,+\, n \,-\, 1 
			 \right) \Ts^{-1}& \text{if} \quad \Ts\le \tau
			\\[0.15 cm]
			2\sigma_w^2
			(1+(n-2)/\kappa)
			\,\Ts & \text{if} \quad \Ts\ge \tau
		\ea
		\right.
	\end{align}
\end{subequations}
where $\tau\DefinedAs(\kappa+1)/2$.
Furthermore, for the gradient noise model ($\sigma_w=\alpha\sigma$), we have
	\begin{subequations}
			\begin{align}\label{eq.thm.mainResultHBLikeJmaxKappaSquareGradNoise}
				J_{\max}
				&\;\le\;
				\sigma^2 \, n \kappa \left((\kappa \,+\, 1)/\Lf^2\right) \Ts^{-1}
				\\[0.cm]
				\label{eq.thm.mainResultHBLikeJminKappaSquareGradNoise}
				J_{\min}
				&\;\le\;
				\sigma^2 \, 2\kappa \left( (\kappa \,+\, 4n\,-\,7)/\Lf^2\right) \Ts^{-1}.
			\end{align}
		\end{subequations}
	\end{mythm}
	
	Theorem~\ref{thm.mainResultHBLike} provides upper bounds on $J_{\max}$ and $J_{\min}$ for a  family of heavy-ball-like algorithms ($\gamma = 0$) parameterized by the settling time $\Ts$. We note that the condition $\Ts \ge (\kappa+1)/2$ in Theorem~\ref{thm.mainResultHBLike} corresponds to non-positive values of the momentum parameter $\beta \leq 0$. For the iterate noise model with $\Ts \le (\kappa+1)/2$ and for the gradient noise model with any settling time, the upper bounds in~Theorem~\ref{thm.mainResultHBLike} scale as $\kappa^2$ for both $J_{\max}\times \Ts$ and $J_{\min}\times \Ts$. This scaling  matches the scaling of the corresponding lower bounds in Theorem~\ref{thm.LowerJ}. Thus, for the gradient noise model, the upper and lower bounds are order-wise tight (in $\kappa$) for any settling time. On the other hand, for the iterate noise model, the lower bounds in Theorem~\ref{thm.LowerJ} are tight only in the accelerated regime $\Ts \leq (\kappa+1)/2$. For this noise model, in the non-accelerated regime $\Ts \geq (\kappa+1)/2$, the alternative lower bounds established in Theorem~\ref{thm.LowerBoundSettlingTimeBased} are tight as they order-wise match the upper bounds in Theorem~\ref{thm.mainResultHBLike}.

	\vsp
	\vsp
	
	\begin{myrem}
		Since $\mathcal{Q}_{\mf}^{\Lf}$ is a subset of the class of $\mf$-strongly convex functions with $\Lf$-Lipschitz continuous gradients, the fundamental lower bounds on  $J_{\max}\times\Ts$ established in Theorem~\ref{thm.LowerJ} carry over to this broader class of problems. Thus, the restriction imposed by the condition number on the tradeoff between settling time and noise amplification goes beyond $\mathcal{Q}_{\mf}^{\Lf}$ and holds for general strongly convex problems.
	\end{myrem}

	\vsp
	
	\begin{myrem}
		The upper bounds in Theorem~\ref{thm.mainResultHBLike} are obtained for a particular choice of constant parameters. Thus, they also provide upper bounds on the best achievable noise amplification bounds 
		$
		J^\star_{\max} \DefinedAs \min_{\alpha,\beta,\gamma} J_{\max}
		$ 
		and
		$
		J^\star_{\min} \DefinedAs \min_{\alpha,\beta,\gamma} J_{\min}
		$
		for a settling time $\Ts$; see Figure~\ref{fig.summary}.
\end{myrem}
	
	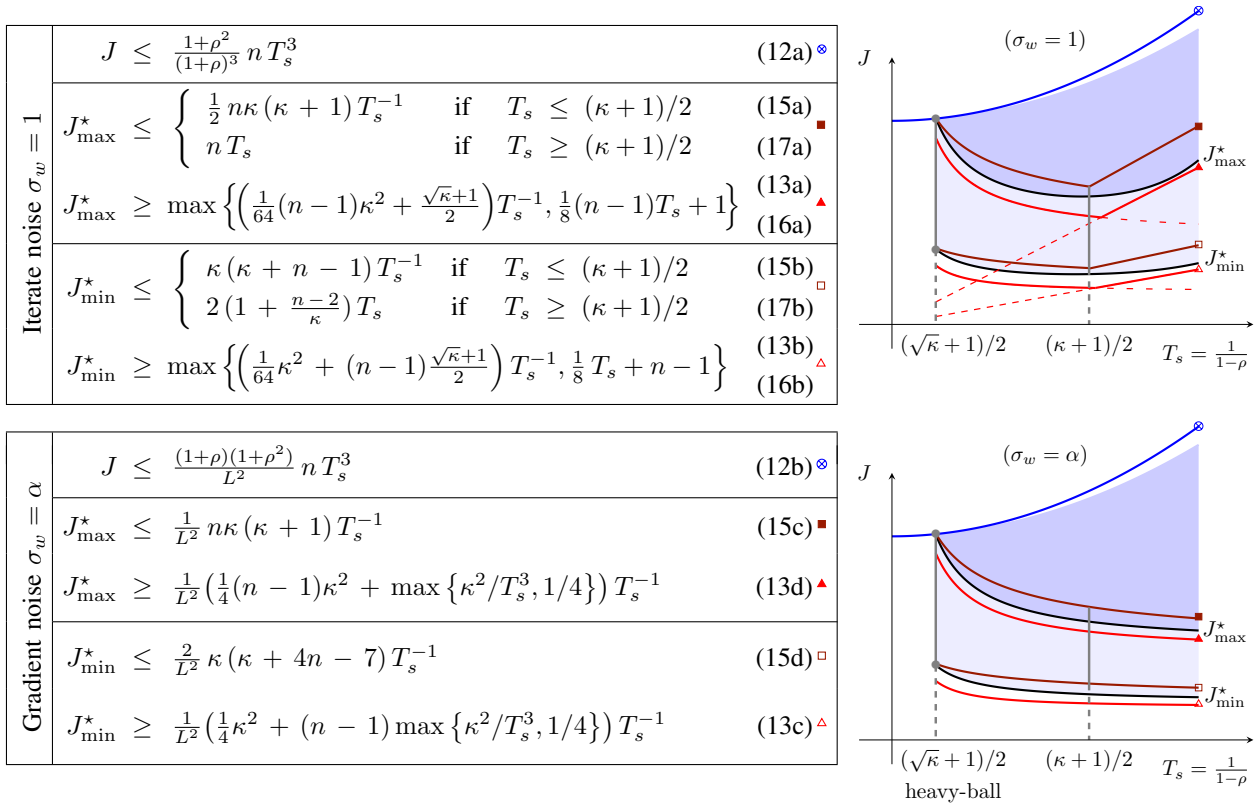
\begin{figure*}
	\centering
	\arrayrulecolor{black}
	\begin{tabular}{cc}
		\begin{tabular}{ !{\color{black}\vline}c!{\color{black}\vline}@{\hspace{0.1 cm}}@{\hspace{-0.1 cm}}r@{\hspace{0.2 cm}}c@{\hspace{0.25 cm}}l|@{\hspace{0.1 cm}}r@{\hspace{0.08 cm}}!{\color{black}\vline}} 
			\hline
			\multirow{12}{0.5em}{\rotatebox{90}{Iterate noise $\sigma_w = 1$}}
			&
			\mystrut{0.45 cm}{0.3 cm}
			$J$ 
			&
			$\le$
			&
			$\frac{1+\rho^2}{(1+\rho)^3} \, n \,\Ts^3$
			&
			\eqref{eq.mainUpperBoundsigma1}\tikz[baseline=-0.9ex]\node[blue,mark size=0.35ex]{\pgfuseplotmark{otimes}};
			\\
			\cline{2-5}
			\mystrut{0.7 cm}{0.4 cm}
			&
			$~ J^{\star}_{\max}$
			&
			$\le$
			&
			$
			\left\{\ba{lr}
			\tfrac{1}{2} \, n \kappa \, (\kappa \,+\, 1)\,\Ts^{-1} & \text{\,~~if~~~ $\Ts\;\le\;(\kappa+1)/2$}
			\\[0.1 cm]
			n\,\Ts &   \text{\,~~if~~~ $\Ts\;\ge\;(\kappa+1)/2$}
			\ea
			\right.
			$
			& 
			\eqref{eq.thm.mainResultHBLikeJmaxLinearGrowth}\tikz[baseline=-0.9ex]\node[color=dred,mark size=0.3ex]{\pgfuseplotmark{square*}};
			\\
			\mystrut{0.6 cm}{0.4 cm}
			&$J^{\star}_{\max}$&$\ge$&
			$\!\max\left\{\!\left(\frac{1}{64}(n-1)\kappa^2
			+
			\frac{\sqrt{\kappa}+1}{2}\right)\!\Ts^{-1}, \frac{1}{8}(n-1)\Ts + 1\!
			\right\}\!\!\!
			$
			&
			\hspace{-0.2 cm}\begin{tabular}{l}
				\eqref{eq.thmLowerJMax}
				\\[-0.15 cm]~~~~~~\tikz[baseline=-0.9ex]\node[color=red,mark size=0.4ex]{\pgfuseplotmark{triangle*}};\\[-0.17 cm]
				\eqref{eq.JmaxLowerBoundSettlingTimeBased}
			\end{tabular}\hspace{-0.2 cm}
			\\
			\cline{2-5} 
			\mystrut{0.7 cm}{0.4 cm}
			&
			$J^{\star}_{\min}$
			&$\le$& 
			$\left\{\ba{lr}
			\kappa\left(
			\kappa 
			\,+\,
			n
			\,-\,1
			\right)\Ts^{-1} & \text{if~~~ $\Ts\;\le\;(\kappa+1)/2$}
			\\[0.1 cm]
			2 \, (1 \, + \, \frac{n\,-\,2}{\kappa}) \, \Ts &   \text{if~~~ $\Ts\;\ge\;(\kappa+1)/2$}
			\ea
			\right.
			$
			&\eqref{eq.thm.mainResultHBLikeJminLinearGrowth}\tikz[baseline=-0.9ex]\node[color=dred,mark size=0.3ex]{\pgfuseplotmark{square}};
			\\
			\mystrut{0.6 cm}{0.4 cm}
			&
			$J^{\star}_{\min}$
			&$\ge$ & $
			\!\max\left\{\!\left(\frac{1}{64}\kappa^2
			\,+\,
			(n-1) 
			\frac{\sqrt{\kappa}+1}{2}\right)\Ts^{-1}, \frac{1}{8} \, \Ts+ n-1
			\right\}\!\!
			$
			&
			\hspace{-0.2 cm}\begin{tabular}{l}
				\eqref{eq.thmLowerJMin}
				\\[-0.15 cm]~~~~~~\tikz[baseline=-0.9ex]\node[color=red,mark size=0.4ex]{\pgfuseplotmark{triangle}};\\[-0.17 cm]
				\eqref{eq.JminLowerBoundSettlingTimeBased}		
			\end{tabular}\hspace{-0.2 cm}
			\\
			\hline
			\arrayrulecolor{white}
			\hline
			\hline
			\hline
			\hline
			\hline
			\arrayrulecolor{black}
			\hline
			&&&&
			\\[-.25cm]
			\multirow{10}{0.5em}{\rotatebox{90}{Gradient noise $\sigma_w=\alpha$}}
			&
			\mystrut{0.3 cm}{0.3 cm}
			$J$ 
			&
			$\le$
			&
			$\frac{(1+\rho)(1+\rho^2)}{\Lf^2}\,n\, \Ts^3$
			&
			\eqref{eq.mainUpperBoundsigmaalpha}\tikz[baseline=-0.9ex]\node[color=blue,mark size=0.4ex]{\pgfuseplotmark{otimes}};
			\\
			\cline{2-5}
			\mystrut{0.5 cm}{0.25 cm}
			&
			$J^\star_{\max}$
		&
		$\le$
		&
		$\frac{1}{\Lf^2} \, n \kappa \, (\kappa \,+\, 1) \,\Ts^{-1}$
		&
		\eqref{eq.thm.mainResultHBLikeJmaxKappaSquareGradNoise}\tikz[baseline=-0.9ex]\node[color=dred,mark size=0.3ex]{\pgfuseplotmark{square*}};
		\\
		\mystrut{0.55 cm}{0.35 cm}
		&$J^\star_{\max}$&$\ge$&
		$\frac{1}{\Lf^2}
		\!
		\left(
		\frac{1}{4}(n\,-\,1) 
		\kappa^2
		\,+\,
		\max\left\{ {\kappa^2}/{\Ts^3}, {1}/{4} \right\}
		\right)\Ts^{-1}$
		&
		\eqref{eq.thmLowerJMax_withAlpha}\tikz[baseline=-0.9ex]\node[color=red,mark size=0.4ex]{\pgfuseplotmark{triangle*}};
		\\
		\cline{2-5} 
		\mystrut{0.6 cm}{0.3 cm}
		&
		$J^{\star}_{\min}$
	&$\le$& 
	$\frac{2}{\Lf^2} \, \kappa \, (\kappa \,+\, 4n \,-\,7) \, \Ts^{-1}$
	&
	\eqref{eq.thm.mainResultHBLikeJminKappaSquareGradNoise}\tikz[baseline=-0.9ex]\node[color=dred,mark size=0.3ex]{\pgfuseplotmark{square}};
	\\
	\mystrut{0.6 cm}{0.4 cm}
	&$J^{\star}_{\min}$&$\ge$ & 
	$ \frac{1}{\Lf^2}
	\!
	\left(
	\frac{1}{4}\kappa^2
	\,+\,
	(n\,-\,1) 
	\max\left\{ {\kappa^2}/{\Ts^3}, {1}/{4} \right\}
	\right)\Ts^{-1}$
	&
	\eqref{eq.thmLowerJMin_withAlpha}\tikz[baseline=-0.9ex]\node[color=red,mark size=0.4ex]{\pgfuseplotmark{triangle}};
	\\
	\hline
\end{tabular}
&\hspace{-0.7 cm}
\begin{tabular}{c}
	\vspace{0.1 cm}
	\\
	\begin{tabular}{c}
		\resizebox{5.8cm}{!}{
			\begin{tikzpicture}
				\node[] (pic) at (0,0) {\includegraphics[]{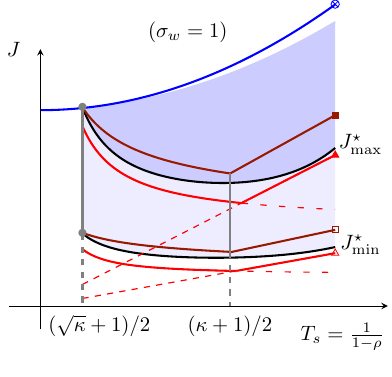}};
			\end{tikzpicture}
		}
	\end{tabular}
	\\
	\begin{tabular}{c}
		\resizebox{5.8cm}{!}{
			\begin{tikzpicture}
				\node[] (pic) at (0,0) {\includegraphics[]{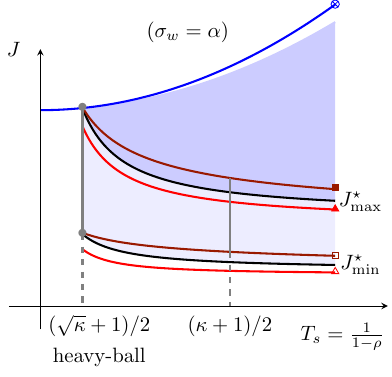}};
			\end{tikzpicture}
		}
	\end{tabular}
\end{tabular}
\end{tabular}	
\caption{Summary of the results established in Theorems~\ref{thm.mainUpperbound}-\ref{thm.mainResultHBLike} for $\sigma^2=1$. The top and bottom rows correspond to the iterate and gradient noise models, respectively, and they illustrate (i)	$
	J^\star_{\max} \DefinedAs \min_{\alpha,\beta,\gamma} \max_{f} J
	$ 
	and
	$
	J^\star_{\min} \DefinedAs \min_{\alpha,\beta,\gamma} \min_{f} J
	$
	subject to a settling time $\Ts$ for $f \in \mathcal{Q}_{\mf}^{\Lf}$ (black curves); and (ii) their corresponding upper (maroon curves) and lower (red curves) bounds in terms of the condition number $ \kappa=\Lf/\mf$, problem size $n$, and settling time $\Ts$. The upper bounds on $J$ established in Theorem~\ref{thm.mainUpperbound} are marked by blue curves. The dark shaded region and its union with the light shaded region respectively correspond to all possible pairs $(\Ts, \max_{f} J)$ and $(\Ts, \min_{f} J)$ for $f \in \mathcal{Q}_{\mf}^{\Lf}$ and any stabilizing parameters ($\alpha,\beta,\gamma$).}
\label{fig.summary}
\end{figure*}
	
	\section{Convergence and noise amplification: geometric characterization}
	\label{sec.modalDecomp}
	
	In this section, we examine the relation between the convergence rate and noise amplification of the two-step momentum algorithm~\eqref{alg.TM_disc} for strongly convex quadratic problems. In particular, the eigenvalue decomposition of the Hessian matrix $Q$ allows us to bring the dynamics into $n$ decoupled second-order systems parameterized by the eigenvalues of $Q$ and the algorithmic parameters ($\alpha,\beta,\gamma$). We utilize the Jury stability criterion to provide a novel geometric characterization of stability and $\rho$-linear convergence and exploit this insight to derive alternative proofs of standard convergence results and quantify fundamental performance tradeoffs. 
	
	\vspace*{-1ex}
	\subsection{Modal decomposition}
	\label{sec.md}
	
	We utilize the eigenvalue decomposition of the Hessian matrix $Q = Q^T \succ 0$, $Q=V\Lambda V^T$, where $\Lambda$ is the diagonal matrix of the eigenvalues and $V$ is the orthogonal matrix of the corresponding eigenvectors. The change of variables $\hat{x}^t \DefinedAs V^T ( x^t - x^\star)$ and $\hat{w}^t \DefinedAs V^T w^t$ allows us to bring~\eqref{eq.two-stepMomentumSS} into $n$ decoupled second-order subsystems,
	\begin{subequations}
		\label{eq.statespace}
		\begin{equation}\label{eq.stateeq}
			\ba{rcl}
			\hat{\psi}_i^{t+1} 
			&\!\!\! =\!\!\!&
			\hat{A}_i \hat{\psi}_i^t
			\,+\,
			\hat{B}_i \hat{w}_i^t
			\\[0.15cm]
			\hat{z}_i^{t}
			&\!\!\! =\!\!\!&
			\hat{C}_i \hat{\psi}_i^t
			\ea
		\end{equation}
		where $\hat{w}_i^{t}$ is the $i$th component of the vector $\hat{w}^t \in \bbR^n$, $\hat{\psi}_i^t = \obt{\hat{x}_i^{t}}{\hat{x}_i^{t+1}}^T$, 
		$\hat{B}_i
		=
		\obt{0}{\sigma_w}^T
		$,
		$
		\hat{C}_i 
		=
		\obt{1}{0}
		$, 
		\beq
		\ba{rcl}
		\hat{A}_i
		&\!\!\! = \!\!\!&
		\hat{A}(\lambda_i)
		\;\DefinedAs\;
		\tbt{0}{1}
		{-a(\lambda_i)}
		{-b(\lambda_i) }
		\ea
		\eeq
		and
		\begin{align}
			a(\lambda)
			\, \DefinedAs \,
			\beta - \gamma \alpha \lambda,
			~
			b(\lambda)
			\, \DefinedAs \,
			(1+\gamma) \alpha \lambda - (1+\beta).
			\label{eq.ab}
		\end{align}
	\end{subequations}
	
	\vspace*{-2ex}
	\subsection{Conditions for linear convergence}
	
	For the class of strongly convex quadratic functions $\mathcal{Q}_{\mf}^{\Lf}$, the best convergence rate $\rho$ is determined by the largest spectral radius of the matrices $\hat{A}(\lambda)$ in~\eqref{eq.statespace} for $\lambda\in[\mf,\Lf]$,
	\beq
	\rho
	\; = \;
	\max_{\lambda \,\in\,[\mf,\Lf]} ~ \spec(\hat{A}(\lambda)).
	\label{eq.rateDecompos}
	\eeq
	For the heavy-ball and Nesterov's accelerated methods, analytical expressions for $\spec(\hat{A}(\lambda))$ were developed and algorithmic parameters that optimize convergences rate were obtained in~\cite{lesrecpac16}. Unfortunately, these expressions do not provide insight into the relation between convergence rates and noise amplification.
	In this paper, we ask the dual question: 
	\bi
	\item {\em For a fixed convergence rate\/} $\rho$, {\em what is the largest condition number\/} $\kappa$ {\em that can be handled by the two-step momentum algorithm\/}~\eqref{alg.TM_disc} {\em with constant parameters?\/} 
	\ei
	We note that the matrices $\hat{A}(\lambda)$ share the same structure as 
	\begin{subequations}
		\begin{equation}
			\label{eq.MatrixM}
			M 
			\, = \,
			\tbt{0}{1}{-a}{-b}
		\end{equation}
		with the real scalars $a$ and $b$ and that the characteristic polynomial associated with the matrix $M$ is given by
		\beq
		F(z)
		\; \DefinedAs \;
		\det \, (zI \, - \, M)
		\; = \;
		z^2 \, + \, b \, z \, + \, a.  
		\label{eq.F(z)}
		\eeq
	\end{subequations}
	We next utilize the Jury stability criterion~\cite[Chap. 4-3]{oga94} to provide conditions for stability of the matrix $M$ given by~\eqref{eq.MatrixM}. 
	
	\vsp
	\begin{mylem}\label{lem.stability}\begin{subequations}
			For the matrix $M \in \bbR^{2 \times 2}$ given by~\eqref{eq.MatrixM}, 
			\begin{align}\label{eq.stabilityCond}
				\spec(M)\,<\,1  
				\;
				\iff 
				\;
				(b,a)
				\,\in\,
				\Delta
			\end{align}
			where the stability set 
			\begin{align}\label{eq.Delta}
				\Delta
				\;\DefinedAs\;	
				\left\{ 
				(b,a) 
				\; | \;
				|b| \, - \, 1 \, < \, a \, < \, 1
				\right\}
			\end{align}
			is an open triangle in the ($b,a$)-plane with vertices
			\begin{align}\label{eq.VerticesDelta}
				X\,=\,(-2,1),\quad Y\,=\,(2,1),\quad Z\,=\,(0,-1).
			\end{align}
		\end{subequations}
	\end{mylem}
	
	\begin{proof}
		The characteristic polynomial $F(z)$ associated with the matrix $M$ is given by~\eqref{eq.F(z)} and the Jury stability criterion~\cite[Chap. 4-3]{oga94} provides necessary and sufficient conditions for stability,
		$
		|a|  <  1
		$,
		$
		F (\pm1)
		=
		1  \pm  b + a
		 >  0.
		$
		The condition $a>-1$ is ensured by the positivity of $F (\pm 1)$.
	\end{proof}
	
	\vsp
	
	For any $\rho > 0$, the spectral radius $\spec(M)$ of the matrix $M$ is smaller than $\rho$ if and only if $\spec(M/\rho)$ is smaller than $1$. This observation in conjunction with Lemma~\ref{lem.stability} allow us to obtain necessary and sufficient conditions for stability with the linear convergence rate $\rho$ of the two-step momentum algorithm~\eqref{alg.TM_disc}.
	
	\vsp
	\begin{mylem}\label{lem.expstability}\begin{subequations}
			For any positive scalar $\rho<1$ and the matrix $M \in \bbR^{2 \times 2}$ given by~\eqref{eq.MatrixM}, we have
			\beq
			\label{eq.expStability}
			\spec(M)\,\le\,\rho  
			\;\iff\;
			(b,a)
			\, \in\, \Delta_\rho
			\eeq
			where the $\rho$-linear convergence set 
			\begin{align}\label{eq.Deltarho}
				\Delta_\rho
				\;\DefinedAs\;	
				\left\{ 
				(b,a) 
				\; | \;
				\rho \, (|b| \, - \, \rho)
				\, \leq \, 
				a 
				\, \leq \, 
				\rho^2
				\right\}
			\end{align}
			is a closed triangle in the ($b,a$)-plane with vertices
			\begin{align}\label{eq.VerticesDeltaRho}
				X_\rho\,=\,(-2\rho,\rho^2),
				~
				Y_\rho\,=\,(2\rho,\rho^2),
				~
				Z_\rho\,=\,(0,-\rho^2).
			\end{align}
		\end{subequations}
	\end{mylem}

	\begin{proof}
		See Appendix~\ref{app.ModalDecomp}.
	\end{proof}
	\begin{figure}[h]
		\begin{center}
			\begin{tabular}{c}
				\resizebox{6.3cm}{!}{
					\begin{tikzpicture}
						\node[] (pic) at (0,0) {\includegraphics[]{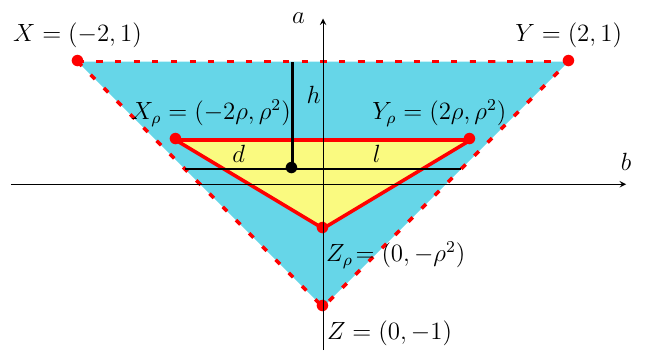}};
					\end{tikzpicture}
				}
			\end{tabular}
		\end{center}
		\vspace*{-0.5cm}
		\caption{The stability set $\Delta$ (the open, cyan triangle) in~\eqref{eq.Delta}  and the $\rho$-linear convergence set $\Delta_\rho$ (the closed, yellow triangle) in~\eqref{eq.Deltarho} along with the corresponding vertices. For $(b,a)$ (black dot) associated with  the matrix $M$ in~\eqref{eq.MatrixM}, the distances  $(d,h,l)$ in~\eqref{eq.dhlDefinition} are marked by black lines.}
		\vspace*{-0.45cm}
		\label{fig.StabilityExpStability-disc}
	\end{figure}
		
	Figure~\ref{fig.StabilityExpStability-disc} illustrates the stability and the $\rho$-linear convergence sets $\Delta$ and $\Delta_\rho$. We note that for any $\rho \in (0,1)$, we have $\Delta_\rho\subset \Delta$. This can be verified by observing that the vertices $(X_\rho,Y_\rho,Z_\rho)$ of $\Delta_\rho$ all lie in $\Delta$.

	\begin{myrem}
			The characterization of the sets $\Delta$ and $\Delta_\rho$ in both continuous and discrete-time settings along with extensions to higher-order systems has been discussed in the literature; e.g., see~\cite{fammed78,blogurmegove12}. In particular, it has been shown that $\Delta_\rho$ is the convex hull of the coefficients associated with the special polynomials $(z-\rho)^2$, $(z+\rho)^2$, and $(z-\rho)(z+\rho)$.
		\end{myrem}
	
	For the two-step momentum algorithm~\eqref{alg.TM_disc}, the functions $a (\lambda)$ and $b (\lambda)$ given by~\eqref{eq.ab} satisfy the affine relation,
	\beq
	(1 \, + \, \gamma)\, a (\lambda) \; + \; \gamma\, b (\lambda)
	\; = \; 
	\beta \, - \, \gamma.
	\label{eq.ab-affine}
	\eeq
	This fact in conjunction with Lemmas~\ref{lem.stability} and~\ref{lem.expstability} allow us to derive conditions for stability and the convergence rate.
	 A similar approach for polynomials of arbitrary degree has been taken in~\cite{blogurmegove12}, where the authors analyzed affine constraints on the coefficients in optimizing the convergence rate. For second order polynomials, we note that the rate of convergence can also be directly characterized as a function of $(a(\lambda),b(\lambda))$. This approach was utilized in~\cite[Lemma~3.1]{cangur22}.
	
	\vsp
	\begin{mylem}\label{lem.expstabilityAlgorithm}
		The two-step momentum algorithm~\eqref{alg.TM_disc} with constant parameters $(\alpha,\beta,\gamma)$ is stable for all functions $ f\in\mathcal{Q}_{\mf}^{\Lf}$ if and only if the following equivalent conditions hold:
		\begin{enumerate}
			\item 
			$
			(b(\lambda),a(\lambda))
			\in 
			\Delta
			$  
			for all
			$
			\lambda \in [\mf,\Lf];
			$
			\item 
			$
			(b(\lambda),a(\lambda))
			\in 
			\Delta
			$  
			for 
			$
			\lambda \in \{ \mf, \Lf \}.
			$
		\end{enumerate}
		Furthermore, the linear convergence rate $\rho<1$ is achieved for all functions $ f\in\mathcal{Q}_{\mf}^{\Lf}$ if and only if the following equivalent conditions hold:
		\begin{enumerate}
			\item 
			$
			(b(\lambda),a(\lambda))
			\in
			\Delta_\rho
			$  
			for all
			$
			\lambda \in [\mf,\Lf];
			$
			\item 
			$
			(b(\lambda),a(\lambda))
			\in
			\Delta_\rho
			$  
			for 
			$
			\lambda \in \{\mf,\Lf\}.
			$
		\end{enumerate}
		Here, $(b(\lambda),a(\lambda))$ is given by~\eqref{eq.ab}, and the stability and $\rho$-linear convergence triangles $\Delta$ and $\Delta_\rho$ are given by~\eqref{eq.Delta} and~\eqref{eq.Deltarho}, respectively.
	\end{mylem}
	
	\vsp
	
	\begin{proof}
		The conditions in 1) follow from combing~\eqref{eq.rateDecompos} with Lemma~\ref{lem.stability} (for stability) and Lemma~\ref{lem.expstability} (for $\rho$-linear convergence). The conditions in 2) follow from the facts that $\Delta$ and $\Delta_\rho$ are convex sets and that $(b(\lambda),a(\lambda))$ is a line segment with end points corresponding to $\lambda=\mf$ and $\lambda=\Lf$.
	\end{proof}
	
	\vsp
	
	Lemma~\ref{lem.expstabilityAlgorithm} exploits the affine relation~\eqref{eq.ab-affine} between $a(\lambda)$ and $b (\lambda)$ and the convexity of the sets $\Delta$ and $\Delta_\rho$ to establish necessary and sufficient conditions for stability and $\rho$-linear convergence: {\em the inclusion of the end points of the line segment\/} $(b(\lambda),a(\lambda))$ {\em associated with the extreme eigenvalues $\mf$ and $\Lf$ of the matrix $Q$ in the corresponding triangle.\/} A similar approach was taken in~\cite[Appendix A.1]{vanles21}, where the affine nature of the conditions resulting from the Jury stability criterion with respect to $\lambda$ was used to conclude that $\spec(\hat{A}(\lambda))$ is a quasi-convex function of $\lambda$ and show that the extreme points $\mf$ and $\Lf$ determine $\spec(A)$. In contrast, we exploit the triangular shapes of the stability and $\rho$-linear convergence sets $\Delta$ and $\Delta_\rho$ and utilize this geometric insight to identify the parameters that optimize the convergence rate and to establish tradeoffs between noise amplification and convergence rate. 
	
	\begin{figure*}[t]
		\centering
		\begin{tabular}{cccc}
			\begin{tabular}{c}
				\resizebox{4.2 cm}{!}{
					\begin{tikzpicture}
						\node[] (pic) at (0,0) {\includegraphics[]{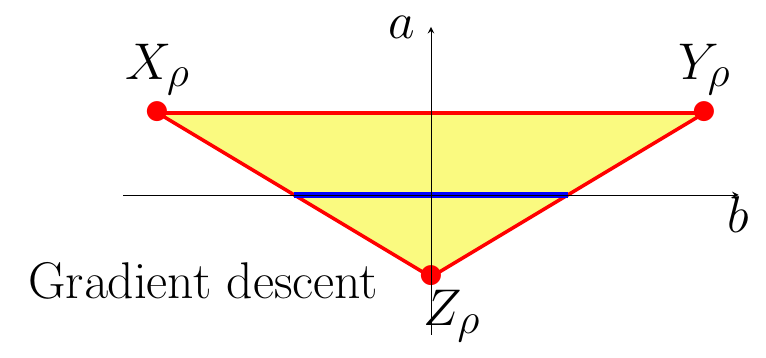}};
					\end{tikzpicture}
				}
			\end{tabular}
			&\hspace{-.8cm}
			\begin{tabular}{c}
				\resizebox{4.2 cm}{!}{
					\begin{tikzpicture}
						\node[] (pic) at (0,0) {\includegraphics[]{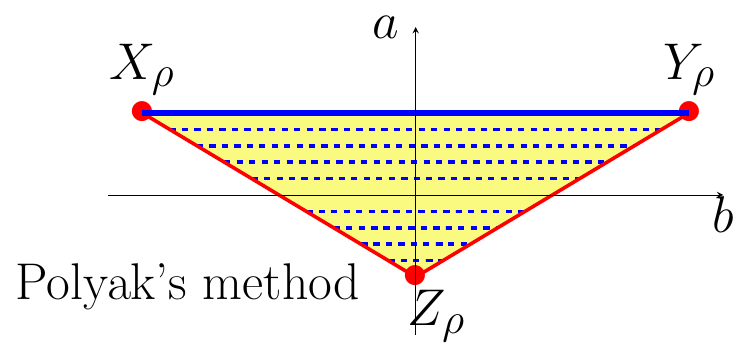}};
					\end{tikzpicture}
				}
			\end{tabular}
			&\hspace{-.8cm}
			\begin{tabular}{c}
				\resizebox{4.2 cm}{!}{
					\begin{tikzpicture}
						\node[] (pic) at (0,0) {\includegraphics[]{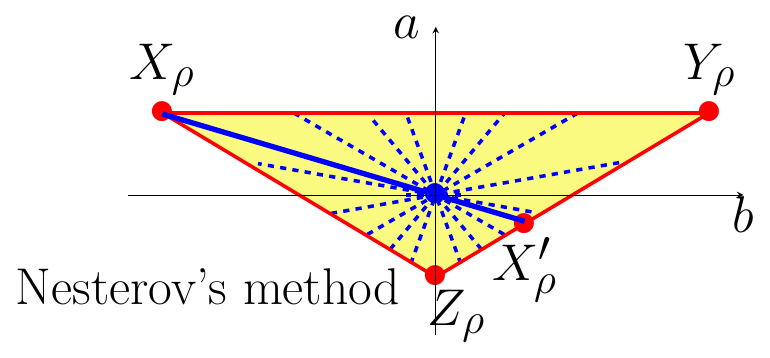}};
					\end{tikzpicture}
				}
			\end{tabular}
		\end{tabular}
		\vspace*{-0.3cm}
		\caption{For a fixed $\rho$-linear convergence triangle $\Delta_\rho$ (yellow), dashed blue lines mark the line segments $(b(\lambda),a(\lambda))$ with $\lambda\in[\mf,\Lf]$ for gradient descent, Polyak's heavy-ball, and Nesterov's accelerated methods as particular instances of the two-step momentum algorithm~\eqref{alg.TM_disc} with constant parameters. The solid blue line segments correspond to the parameters for which the algorithm achieves rate $\rho$ for the largest possible condition number given by~\eqref{eq.optimalRatesAltProof}.}
		\vspace*{-0.5cm}
		\label{fig.GDNAHB-structure}
	\end{figure*}
	The following corollary is immediate.
	
	\vsp
	
	\begin{mycor}\label{cor.onerateforall}
		Let the two-step momentum algorithm~\eqref{alg.TM_disc} with constant parameters $(\alpha,\beta,\gamma)$ minimize a function $f\in\mathcal{Q}_{\mf}^{\Lf}$ with a linear rate $\rho<1$. Then, the convergence rate $\rho$ is achieved for all functions $f \in \mathcal{Q}_{\mf}^{\Lf}$.
	\end{mycor}
	\vsp
	\begin{proof}
		Lemma~\ref{lem.expstabilityAlgorithm} implies that only the extreme eigenvalues $\mf$ and $\Lf$ of $Q$ determine $\rho$. Since all functions $f \in \mathcal{Q}_{\mf}^{\Lf}$ share the same extreme eigenvalues, this completes the proof. 
	\end{proof}
	
	\vsp
	
	For the two-step momentum algorithm~\eqref{alg.TM_disc} with constant parameters, Lemma~\ref{lem.expstabilityAlgorithm} leads to a a simple alternative proof for the fundamental lower bound~\eqref{eq.NesFunLow} on the settling time established by Nesterov. Our proof utilizes the fact that for any point $(b(\lambda),a(\lambda)) \in \Delta_\rho$, the horizontal signed distance to the edge $XZ$ of the stability triangle $\Delta$ satisfies
	\begin{align}\label{eq.dDef}
		d(\lambda)
		\;\DefinedAs\;
		a(\lambda)\,+\,b(\lambda)\,+\,1
		\;=\;
		\alpha\lambda.
	\end{align}
	where $a$ and $b$ are given by~\eqref{eq.ab}; see Figure~\ref{fig.StabilityExpStability-disc}.
	
	\begin{myprop}\label{prop.recoverNesLowerbound}
		Let the two-step momentum algorithm~\eqref{alg.TM_disc} with constant parameters $(\alpha,\beta,\gamma)$ achieve the convergence rate $\rho<1$ for all functions $f\in\mathcal{Q}_{\mf}^{\Lf}$. Then, lower bound~\eqref{eq.NesFunLow} on the settling time holds and it is achieved by the heavy-ball method with the parameters provided in Table~\ref{tab:rates}.
		\end{myprop}	
	\vsp
	\begin{proof}
		Let $d(\mf)=\alpha \mf$ and $d(\Lf)=\alpha\Lf$ denote the values of the function $d(\lambda)$ associated with the points~$(b(\mf),a(\mf))$ and~$(b(\Lf),a(\Lf))$, where $(b,a)$ and $d$ are given by~\eqref{eq.ab} and~\eqref{eq.dDef}, respectively. Lemma~\ref{lem.expstabilityAlgorithm} implies that $(b(\Lf),a(\Lf))$ and $(b(\mf),a(\mf))$ lie in the $\rho$-linear convergence triangle $\Delta_\rho$. Thus, 
		\begin{align}\label{eq.dmaxdminProof}
			{d_{\max}}/{d_{\min}}
			\;\ge\;
			{d(\Lf)}/{d(\mf)}
			\;=\;
			\kappa
		\end{align}
		where $d_{\max}$ and $d_{\min}$ are the largest and smallest values of $d$ among all points $(b,a)\in\Delta_\rho$. From the shape of $\Delta_\rho$, we conclude that $d_{\max}$ and $d_{\min}$ correspond to the vertices $Y_\rho$ and $X_\rho$ of $\Delta_\rho$ given by~\eqref{eq.VerticesDeltaRho}; see Figure~\ref{fig.StabilityExpStability-disc}. Thus,
		\begin{subequations}
			\begin{align}\label{eq.Dmax}
				d_{\max}
				&\;=\;
				d_{Y_\rho}
				\;=\;
				1\,+\,\rho^2\,+\,2\,\rho\;=\; (1\,+\,\rho)^2
				\\[0.cm]
				\label{eq.Dmin}
				d_{\min}
				&\;=\;
				d_{X_\rho}
				\;=\;
				1\,+\,\rho^2\,-\,2\,\rho\;=\; (1\,-\,\rho)^2.
			\end{align}\label{eq.DmaxDmin}\end{subequations}
Combining~\eqref{eq.dmaxdminProof} with~\eqref{eq.DmaxDmin} yields
		\begin{align}\label{eq.boudnKappaGeneral}
			\kappa
			\;=\;
			\dfrac{d(\Lf)}{d(\mf)}
			\;\le\;
			\dfrac{d_{\max}}{d_{\min}}
			\;=\;
			\dfrac{( 1\,+\,\rho )^2}{( 1\,-\,\rho )^2}.
		\end{align}
		Rearranging terms in~\eqref{eq.boudnKappaGeneral} gives lower bound~\eqref{eq.NesFunLow}. 	
		\end{proof}
	
	\vsp
	
	To provide additional insight, we next examine the implications of Lemma~\ref{lem.expstabilityAlgorithm} for gradient descent, Polyak's heavy-ball, and Nesterov's accelerated algorithms. In all three cases, our dual approach recovers the optimal convergence rates provided in Table~\ref{tab:rates}. From the affine relation~\eqref{eq.ab-affine}, it follows that $(b(\lambda),a(\lambda))$ with $\lambda\in[\mf,\Lf]$ for,
	\begin{itemize}
		\item gradient descent ($\beta=\gamma=0$), is a horizontal line segment parameterized by $a (\lambda) =0$;
		\item heavy-ball method ($\gamma=0$), is a horizontal line segment parameterized by $a (\lambda) =\beta$;
		and
		\item Nesterov's accelerated method ($\beta=\gamma$), is a line segment parameterized by
		$
		a (\lambda)
		= 
		- \beta b (\lambda) / (1 + \beta).
		$
	\end{itemize}
		These observations are illustrated in Figure~\ref{fig.GDNAHB-structure} and, as we show in the proof of Lemma~\ref{lem.expstabilityAlgorithm}, to obtain the largest possible condition number for which the convergence rate $\rho$ is feasible for each algorithm, one needs to find the largest ratio
	$
	d(\Lf)/d(\mf)=\kappa
	$
	among all possible orientations for the line segment $(b(\lambda),a(\lambda))$ with $\lambda\in[\mf,\Lf]$ to lie within $\Delta_\rho$. 
	\begin{itemize}
		\item For gradient descent, the largest ratio $d(\Lf)/d(\mf)$ corresponds to the intersections of the horizontal axis and the edges $Y_\rho Z_\rho$ and $X_\rho Z_\rho$ of the triangle $\Delta_\rho$, which are given by $(\rho,0)$ and $(-\rho,0)$, respectively. Thus,
		\begin{subequations}\label{eq.optimalRatesAltProof}
			\begin{align}
				\kappa
				\;=\;
				d(\Lf)/d(\mf)
				\;\le\;
				(1\,+\,\rho)/(1\,-\,\rho).
				\label{eq.gd-rate}
			\end{align}
			Rearranging terms in~\eqref{eq.gd-rate} yields a lower bound on the settling time for gradient descent
			$
			1/(1 - \rho)
			\ge
			(\kappa + 1)/2.
			$
			This lower bound is tight as it can be achieved by choosing the parameters in Table~\ref{tab:rates}, which place $(b(\lambda),a(\lambda))$ to $(\rho,0)$ and $(-\rho,0)$ for $\lambda=\Lf$ and $\lambda=\mf$, respectively. 
						
			\item For the heavy-ball method, the optimal rate is recovered by designing the parameters ($\alpha,\beta$) such that the vertices $X_\rho$ and $Y_\rho$ belong to the horizontal line segment $(b(\lambda),a(\lambda))$,
			\begin{align}
				\kappa
				\;=\;
				d(\Lf) / d(\mf)
				\;\le\;
				(1\,+\,\rho)^2 / (1\,-\,\rho)^2.
			\end{align}
			By choosing $d (\Lf) = d_{Y_\rho}$ and $d (\mf) = d_{X_\rho}$, we recover the optimal parameters provided in Table~\ref{tab:rates} and achieve the fundamental lower bound~\eqref{eq.NesFunLow} on the convergence rate.
			\item For Nesterov's accelerated method, the largest ratio $d(\Lf)/d(\mf)$ corresponds to the line segment $X_\rho X'_\rho$ that passes through the origin, where $X'_\rho=(2\rho/3,-\rho^2/3)$ lies on the edge $Y_\rho Z_\rho$; see Appendix~C in supplementary material. Thus,
			\begin{align}\label{eq.optimalRateNesProof}
				\hspace{-0.25 cm}
				\kappa
				\,=\,
				{d(\Lf)}/{d(\mf)}
				\,\le\,
				(1 + \rho)(3 - \rho)/ (3 (1-\rho)^2).
			\end{align}
		\end{subequations}
		Rearranging terms in this inequality provides a lower bound on the settling time
		$
		1/(1 - \rho)
		\ge 
		\sqrt{3\kappa + 1}/2.
		$
		This lower bound is tight and it can be achieved with the parameters provided in Table~\ref{tab:rates}, which place $(b(\Lf),a(\Lf))$ to $X'_\rho$ and $(b(\mf),a(\mf))$ to $X_\rho$. 
	\end{itemize} 
	\vsp
	Figure~\ref{fig.GDNAHB-structure} illustrates the optimal orientations discussed above.
		
	\vspace*{-2ex}
\subsection{Noise amplification}	
	
	To quantify the noise amplification of the two-step momentum algorithm~\eqref{alg.TM_disc}, we utilize an alternative characterization of the stability triangle $\Delta$. As illustrated in Figure~\ref{fig.StabilityExpStability-disc}, let $d$ and $l$ denote the horizontal signed distances of the point ($a,b$) to the edges $XZ$ and $YZ$,
	\begin{subequations}
		\label{eq.dhlDefinition}
		\beq
		\ba{rcl}
		d (\lambda)
		& \!\! \DefinedAs \!\! &
		a (\lambda) \,+\, b (\lambda) \,+\, 1
		\\[0.1cm]
		l (\lambda)
		& \!\! \DefinedAs \!\! &
		a (\lambda) \,-\,b (\lambda) \, + \,1
		\ea
		\eeq
		and let $h$ denote its vertical signed distance to the edge $XY$, 
		\beq
		h (\lambda)
		\,\DefinedAs\,
		1 \, - \, a (\lambda).
		\eeq
	\end{subequations}
Then, the equivalence condition
		\begin{align}\label{eq.equivalenceDelta}
			(b,a)
			\,\in\,
			\Delta
			&\,\iff\,
			h,\,d,\,l
			\, > \, 
			0
		\end{align}
follows from the definition of the set $\Delta$ in~\eqref{eq.Delta}.
		
	While analytical expressions for $J$ in terms of the spectrum of the Hessian matrix and the algorithmic parameters have been derived in the literature (e.g., see~\cite[Theorem~1]{mohrazjovCDC18} for gradient decent, heavy-ball, and Nesterov's accelerated algorithms, and~\cite[Appendix~A]{aybfalgurozd19b} and~\cite[Lemma~A.1]{cangur22} for the general case), the novelty of Theorem~\ref{thm.varianceJhat} lies in expressing $J$ in terms of the reciprocals of the distances $d (\lambda_i)$, $h (\lambda_i)$, and $l (\lambda_i)$ of the point $(b(\lambda_i),a(\lambda_i))$ to the edges of the stability triangle for the noisy two-step momentum algorithm~\eqref{alg.TM_disc}. This geometric insight facilitates proofs of our main results. The proof of Theorem~\ref{thm.varianceJhat} is straightforward and it is omitted for brevity.
	\vsp
	\begin{mythm}
		\label{thm.varianceJhat}
		For a strongly convex quadratic objective function $f\in\mathcal{Q}_{\mf}^{\Lf}$ with the Hessian matrix $Q$, the steady-state variance of $x^t-x^\star$ for the two-step momentum algorithm~\eqref{alg.TM_disc} with any stabilizing parameters $(\alpha,\beta,\gamma)$ is determined~by
		\begin{align*}
			J
			\;=\;
			\sum_{i \, = \, 1}^n 
			\dfrac{ \sigma_w^2}{2 h(\lambda_i)} 
			\!
			\left(
			\dfrac{1}{l (\lambda_i)} \, + \, \dfrac{1}{d (\lambda_i)}
			\right)
			\; \AsDefined \;
			\sum_{i \, = \, 1}^n \hat{J}(\lambda_i)
		\end{align*}
		Here, $\hat{J}(\lambda_i)$ denotes the modal contribution of the $i$th eigenvalue $\lambda_i$ of $Q$ to the steady-state variance, ($d,h,l$) are defined in~\eqref{eq.dhlDefinition}, and ($a,b$) are given by~\eqref{eq.ab}.
	\end{mythm}
	
	\vsp
	In Appendix~\ref{app.LyapuEquation}, we describe how the algebraic Lyapunov equation for the steady-state covariance matrix of the error in the optimization variable can be used to compute the noise amplification $J$. Theorem~\ref{thm.varianceJhat} demonstrates that $J$ depends on the entire spectrum of the Hessian matrix $Q$ and not only on its extreme eigenvalues $m$ and $L$, which determine the convergence rate. Since for any $f \in \mathcal{Q}_{\mf}^{\Lf}$ the extreme eigenvalues of $Q$ are fixed at $m$ and $L$, we have 
		\beq
		\ba{rcl}
		J_{\max} 
		& \!\!\! = \!\!\! &
		\hat{J}(\mf) 
		\,+\,
		\hat{J}(\Lf)
		\,+\,
		(n\,-\,2) \hat{J}_{\max}
		\\[0.1cm]
		J_{\min} 
		& \!\!\! = \!\!\! &
		\hat{J}(\mf) 
		\,+\,
		\hat{J}(\Lf)
		\,+\,
		(n\,-\,2) \hat{J}_{\min}
		\ea
		\hspace{-0.15 cm}
		\label{eq.JhatMinMaxDetermineJMinMax}
		\eeq 
		where
		$
		\hat{J}_{\max}
		\DefinedAs 
		\max_{\lambda\in[\mf,\Lf]} \hat{J}(\lambda)
		$,
		$
		\hat{J}_{\min}
		\DefinedAs
		\min_{\lambda\in[\mf,\Lf]} \hat{J}(\lambda)
		$.
	We use~\eqref{eq.JhatMinMaxDetermineJMinMax} to determine explicit bounds on $J_{\max}$ and $J_{\min}$  in terms of the condition number and the settling time.
	
	\vspace*{-1ex}
	\section{Designing order-wise Pareto-optimal algorithms with adjustable parameters}
	\label{sec.Tuning}
	
	We now utilize the geometric insight developed in Section~\ref{sec.modalDecomp} to design algorithmic parameters that tradeoff settling time and noise amplification. In particular, we introduce two parameterized families of heavy-ball-like ($\gamma = 0$) and Nesterov-like ($\gamma = \beta$) algorithms that provide {\em continuous transformations\/} from gradient descent to the corresponding accelerated algorithm (with the optimal convergence rate) via a homotopy path parameterized by the settling time $\Ts$. For both the iterate and gradient noise models, we establish an order-wise tight scaling $\Theta(\kappa^2)$ for $J_{\max}\times \Ts$ and  $J_{\min}\times \Ts$ in accelerated regime (i.e., when $\Ts$ is smaller than the settling time of gradient descent with the optimal stepsize, $(\kappa + 1)/2$). This is a direct extension of~\cite[Theorem~4]{mohrazjovTAC21} which studied gradient descent and its accelerated variants for the parameters that optimize the corresponding convergence rates. 
	
	We also examine performance tradeoffs for the parameterized family of heavy-ball-like algorithms with negative momentum parameter $\beta<0$. This decelerated regime corresponds to settling times larger than $(\kappa+1)/2$ and it captures a key difference between the two noise models: {\em for $\Ts\ge(\kappa+1)/2$, $J_{\max}$ and $J_{\min}$ grow linearly with the settling time $\Ts$ for the iterate noise model and they remain inversely proportional to $\Ts$ for the gradient noise model.\/} Comparison with the lower bounds in Theorems~\ref{thm.LowerJ} and~\ref{thm.LowerBoundSettlingTimeBased} shows that the parameterized family of heavy-ball-like methods yields order-wise optimal (in $\kappa$ and $\Ts$) $J_{\max}$ and $J_{\min}$ for both noise models. The results presented here prove all upper bounds in Theorems~\ref{thm.LowerBoundSettlingTimeBased} and~\ref{thm.mainResultHBLike}. 
	
	\vspace*{-2ex}
	\subsection{Parameterized family of heavy-ball-like methods}
	
	For the two-step momentum algorithm~\eqref{alg.TM_disc} with $\gamma=0$, the line segment $(b(\lambda),a(\lambda))$ parameterized by $\lambda\in[\mf,\Lf]$ is parallel to the $b$-axis in the ($b,a$)-plane and it satisfies $a (\lambda) = \beta$. As described in Section~\ref{sec.modalDecomp}, gradient descent and heavy-ball methods with the optimal parameters provided in Table~\ref{tab:rates} are obtained for $\beta = 0$ and $\beta = \rho^2$, respectively, and the corresponding end points $(b(\mf),a(\mf))$ and $(b(\Lf),a(\Lf))$ lie at the edges $X_\rho Z_\rho$ and $Y_\rho Z_\rho$ of the $\rho$-linear convergence triangle $\Delta_\rho$. Inspired by this observation, we propose a family of parameters for which $\beta = c \rho^2$, for some scalar $c\in[-1,1]$, and determine the stepsize $\alpha$ such that the above end points lie at $X_\rho Z_\rho$ and $Y_\rho Z_\rho$,
	\beq
	\label{eq.parametersHeavy-ballLike}
	\alpha
	\;=\;
	(1\,+\,\rho)(1\,+\,c\rho) / \Lf,
	\;\;
	\beta
	\;=\;
	c\rho^2,
	\;\;
	\gamma
	\;=\;
	0.
	\eeq 
	This yields a continuous transformation between the standard heavy-ball method ($c=1$) and gradient descent ($c=0$) for a fixed condition number $\kappa$. In addition, the momentum parameter $\beta$ in~\eqref{eq.parametersHeavy-ballLike} becomes negative for $c<0$; see Figure~\ref{fig.GDNAHB-structure} for an illustration. In Lemma~\ref{lem.individualJhatHeavyBallLike}, we provide expressions for the scalar $c$ as well as for $\hat{J}_{\max}$ and $\hat{J}_{\min}$ defined in~\eqref{eq.JhatMinMaxDetermineJMinMax} in terms of the condition number $\kappa$ and the convergence rate $\rho$.
	
	\vsp
	
	\begin{mylem}\label{lem.individualJhatHeavyBallLike}
		For the class of functions $\mathcal{Q}_{\mf}^{\Lf}$ with the condition number $\kappa=\Lf/\mf$, let the scalar $\rho$ be such that   	\beq
		\Ts \,=\, 1/(1 \, - \, \rho) \,\ge\, (\sqrt{\kappa} \, +1 \, )/2.
		\non
		\eeq
		Then, the two-step momentum algorithm~\eqref{alg.TM_disc} with parameters~\eqref{eq.parametersHeavy-ballLike}
		achieves the convergence rate $\rho$, and the largest and smallest values $\hat{J}_{\max}$ and $\hat{J}_{\min}$ of $\hat{J}(\lambda)$ satisfy 
		\begin{align*}
			\hat{J}_{\max}
			&\,=\,
			\hat{J}(\mf)
			\,=\,
			\hat{J}(\Lf)
			\,=\,
			\dfrac{\sigma_w^2(\kappa\,+\,1)}{2(1-c\rho^2)(1+\rho)(1+c\rho)}
			\\[0.cm]
			\hat{J}_{\min}
			&\;=\;
			\hat{J}(\hat{\lambda})
			\;=\;
			\dfrac{\sigma_w^2}{( 1 \, + \, c \rho^2 ) (1 \, - \, c \rho^2)}
		\end{align*}
		where $\hat{\lambda}\DefinedAs(\mf+\Lf)/2$ and the scalar $c$ is given by
		\begin{align}\label{eq.scalar_c_HeavyballLike}
			c
			\;\DefinedAs\;
			\dfrac{\kappa \,-\, (1\,+\,\rho) / (1 \,-\, \rho)}{\rho\left(\kappa \,+\, (1\,+\,\rho) / (1 \,-\, \rho) \right)}
			\, \in \,
			[-1,\,1].
		\end{align}
	\end{mylem}
	\begin{proof}
		See Appendix~\ref{app.Tuning}.
	\end{proof}
	
	\vsp
	The parameters in~\eqref{eq.parametersHeavy-ballLike} with $c$ given by~\eqref{eq.scalar_c_HeavyballLike} are equivalent to the parameters presented in Theorem~\ref{thm.mainResultHBLike}. Lemma~\ref{lem.individualJhatHeavyBallLike} in conjunction with~\eqref{eq.JhatMinMaxDetermineJMinMax} allow us to compute $J_{\max}$ and $J_{\min}$.
	
	\vsp
	
	\begin{mycor}\label{cor.JHeavyBallLike}
		The parameterized family of heavy-ball-like methods~\eqref{eq.parametersHeavy-ballLike} satisfies
		$
		{J_{\max}}
		=
		n \hat{J}(\mf)
		=
		n \hat{J}(\Lf)
		$
		and
		$
		{J_{\min}}
		=
		2 \hat{J}(\mf)
		+
		(n-2)\hat{J}(\hat{\lambda})
		$,
		where $\hat{J}(\mf)$ and $\hat{J}(\hat{\lambda})$ are given in Lemma~\ref{lem.individualJhatHeavyBallLike}, and
		$J_{\max}$ and $J_{\min}$  defined in~\eqref{eq.JmaxJminDef} are the largest and smallest values of $J$  when the algorithm is applied to $f\in\mathcal{Q}_{\mf}^{\Lf}$ with the condition number $\kappa=\Lf/\mf$.
	\end{mycor}
	
	\vsp
		
	 Proposition~\ref{prop.JHeavyBallLikeBounds} uses the expressions in Corollary~\ref{cor.JHeavyBallLike} to establish order-wise tight upper and lower bounds on $J_{\max}$ and $J_{\min}$ for the parameterized family of heavy-ball-like algorithms~\eqref{eq.parametersHeavy-ballLike}. Our upper and lower bounds are within constant factors of each other and they are expressed in terms of the problem size $n$, condition number $\kappa$, and settling time $\Ts$.
	
	\vsp
	
		\begin{myprop}\label{prop.JHeavyBallLikeBounds}
		For the parameterized family of heavy-ball-like methods~\eqref{eq.parametersHeavy-ballLike},     $J_{\max}$ and $J_{\min}$ in~\eqref{eq.JmaxJminDef} satisfy,
		\begin{subequations}
			\begin{align}
				J_{\max}\times \Ts
				&\;=\;
				\sigma_w^2 \, p_{1c} (\rho) \,n\,\kappa(\kappa \,+\, 1)
				\\[0.cm]
				J_{\min}
				\times\Ts
				&\;=\;
				\sigma_w^2 \, \kappa \left(2\, p_{1c} (\rho) \, (\kappa \,+\, 1) 
				\,+\,
				(n \, - \, 2) \, p_{2c} (\rho)
				\right).
			\end{align}
		\end{subequations}
		Furthermore, for the gradient noise model ($\sigma_w=\alpha\sigma$), 
		\begin{subequations}
			\begin{align}
				J_{\max}\times\Ts
				&\;=\;
				\sigma^2  p_{3c} (\rho) \,n\,\kappa(\kappa \,+\, 1)
				\\[0.cm]
				J_{\min}\times\Ts
				&\;=\;
				\sigma^2 \kappa \left(2\, p_{3c} (\rho) \, (\kappa \,+\, 1) 
				\,+\,
				(n \, - \, 2) \, p_{4c} (\rho)
				\right)
			\end{align}
		\end{subequations}
		where	
		\be
		\ba{rcl}
		p_{1c} (\rho) 
		& \!\!\! \DefinedAs \!\!\! & 
		q_c (\rho) / (2(1+\rho)^2(1+c\rho)^2)
		\\[0.1cm]
		p_{2c} (\rho) 
		& \!\!\! \DefinedAs \!\!\! &
		q_c (\rho) / ((1+\rho)(1+c\rho^2)(1+c\rho))
		\\[0.1cm]
		p_{3c} (\rho) 
		& \!\!\! \DefinedAs \!\!\! & 
		q_c (\rho)/(2 L^2)
		\\[0.1cm]
		p_{4c} (\rho) 
		& \!\!\! \DefinedAs \!\!\! & 
		q_c (\rho) q_{-c}(\rho) (1+\rho)/L^2 
		\\[0.1cm]
		q_c (\rho)
		& \!\!\! \DefinedAs \!\!\! & 
		(1 \, - \, c \rho)/(1 \, - \, c \rho^2).
		\ea
		\label{eq.p_iDef}
		\ee
In addition, for $c\in[0,1]$, 
		$
		p_{1c} (\rho)
		\in
		[1/64, 1/2]
		$ 
		and
		$
		p_{2c} (\rho)
		\in
		[1/16,1]
		$;
		and for $c\in[-1,1]$, 
		$
		p_{3c} (\rho)
		\in
		[1/(4\Lf^2),1/\Lf^2]
		$
		and
		$
		p_{4c} (\rho)
		\in
		[1/(4\Lf^2),4/\Lf^2].
		$
	\end{myprop}
	\vsp
	\begin{proof}
		See Appendix~\ref{app.Tuning}.
	\end{proof}
	\vsp
	\begin{myprop}
		\label{prop.HBlikeNegative}
		For the parameterized family of heavy-ball-like methods~\eqref{eq.parametersHeavy-ballLike} with $c\in[-1,0]$,     $J_{\max}$ and $J_{\min}$ in~\eqref{eq.JmaxJminDef} satisfy,
		\begin{subequations}
			\begin{align}
				J_{\max}
				&\;=\;
				\sigma_w^2 \, p_{5c}(\rho)\,n\,(1 \,+\, 1/\kappa)\,\Ts
				\\[0.cm]
				J_{\min}
				&\;=\;
				\sigma_w^2 \big(2\,p_{5c}(\rho)\,(1 \, + \, 1/\kappa) \, + \, p_{6c}(\rho)(n-2)/\kappa \big) \, \Ts		
			\end{align}
		\end{subequations}
		where
		$
		p_{5c}(\rho)
		\DefinedAs
		1/({2(1+ |c| \rho)(1+|c|\rho^2)})
		\in
		[1/8,1/2]
		$
		and
		$
		p_{6c}(\rho)
		\DefinedAs
		2 (1+\rho) p_{5c}(\rho)q_{-c}(\rho)
		\in
		[1/8,2].
		$
	\end{myprop}
	\vsp
	\begin{proof}
		See Appendix~\ref{app.Tuning}.
	\end{proof}
	\vsp
	
	The upper bounds in Theorems~\ref{thm.mainResultHBLike} and~\ref{thm.LowerBoundSettlingTimeBased} follow from Propositions~\ref{prop.JHeavyBallLikeBounds} and~\ref{prop.HBlikeNegative}, respectively. Since these upper bounds have the same scaling as the corresponding lower bounds in Theorems~\ref{thm.LowerJ} and~\ref{thm.LowerBoundSettlingTimeBased} that hold for all stabilizing parameters ($\alpha,\beta,\gamma$), this demonstrates tightness of lower bounds for all settling times and for both noise models.
	
		\vspace*{-2ex}
	\subsection{Parameterized family of Nesterov-like methods}
	
	For the two-step momentum algorithm~\eqref{alg.TM_disc} with $\gamma=\beta$, the line segment $(b(\lambda),a(\lambda))$ parameterized by $\lambda\in[\mf,\Lf]$  passes through the origin. As described in Section~\ref{sec.modalDecomp}, gradient descent and Nesterov's method  with the optimal parameters provided in Table~\ref{tab:rates} are obtained for $a=0$ and $a= - (\rho/2) b$, respectively, and the corresponding end points $(b(\mf),a(\mf))$ and $(b(\Lf),a(\Lf))$ lie on the edges $X_\rho Z_\rho$ and $Y_\rho Z_\rho$ of the $\rho$-linear convergence triangle $\Delta_\rho$. To provide a continuous transformation between these two standard algorithms, we introduce a parameter $c\in[0,1/2]$, and let  the line segment satisfy $a (\lambda) = - c\rho b (\lambda)$ and take its end points at the edges  $X_\rho Z_\rho$ and $Y_\rho Z_\rho$; see Figure~\ref{fig.GDNAHB-structure} for an illustration. This can be accomplished with the following choice of parameters,
	\be\label{eq.parametersNesterovLike}
	\ba{rcl}
	\alpha
	&\!\!\!=\!\!\!&
	(1\,+\,\rho)(1\,+\,c\,-\,c\rho)/
	( \Lf(1\,+\,c) )
	\\[0.15 cm]
	\gamma
	&\!\!\!=\!\!\!&
	\beta
	\,=\,
	{c\rho^2}/
	( (\alpha\Lf\,-\,1)(1\,+\,c) ).
	\ea
	\ee
	For the parameterized family of Nesterov-like algorithms~\eqref{eq.parametersNesterovLike}, Proposition~\ref{prop.JNesterovLikeBounds} establishes the settling time and characterizes the dependence of $J_{\min}\times \Ts$ and $J_{\max}\times\Ts$ on the condition number $\kappa$ and the problem size $n$. 
	
	\vsp
	
	\begin{myprop}
		\label{prop.JNesterovLikeBounds}
		For the class of functions $\mathcal{Q}_{\mf}^{\Lf}$ with the condition number $\kappa=\Lf/\mf$, let the scalar $\rho$ be such that   
		\[
		\Ts\;=\;1/(1-\rho)\;\in\;[(\sqrt{3\kappa+1})/2,(\kappa+1)/2].
		\]
		The two-step momentum algorithm~\eqref{alg.TM_disc} with parameters~\eqref{eq.parametersNesterovLike}
		achieves the convergence rate $\rho$ and 
		satisfies
		\be
		\ba{rcl}
		J_{\max}\times \Ts
		&\!\!\! \ge \!\!\!\!&
		\sigma_w^2
		\left((n-1) \kappa( \kappa \,+\, 1)/32 
		\,+\, 
		\sqrt{3\kappa+1}/2
		\right)
		\\[0.1cm]
		J_{\max}\times \Ts
		&\!\!\! \le \!\!\!\!&
		6 \,\sigma_w^2
		n \, \kappa(3 \kappa \,+\, 1)
		\\[0.1cm]
		J_{\min}\times \Ts
		&\!\!\! \ge \!\!\!\!&
		\sigma_w^2\left( \kappa(\kappa \,+\, 1)/32 
		+
		(n
		-
		1)
		\sqrt{3\kappa+1}/2
		\right)
		\\[0.1cm]
		J_{\min}\times \Ts
		&\!\!\!\le\!\!\!\!&
		\sigma_w^2\left(6\,\kappa(3\kappa \,+\, 1) 
		+
		(n
		-
		1)
		(\kappa+1)/2
		\right)
		\ea
		\non
		\ee
		where $J_{\max}$ and $J_{\min}$ are the largest and smallest values that $J$ can take when the algorithm is applied to $f \in \mathcal{Q}_{\mf}^{\Lf}$ with the condition number $\kappa=\Lf/\mf$, and the scalar $c\in[0,1/2]$ is the solution to the quadratic equation
		\begin{align*}
			\kappa (1-\rho)(1 - c \rho - c^2(1+\rho))
			=
			(1+\rho)(1 - c \rho - c^2(1-\rho)).
		\end{align*}
	\end{myprop}
	\begin{proof}
	See Appendix~D in supplementary material.
	\end{proof}
	\vsp
	Since $\alpha$ in~\eqref{eq.parametersNesterovLike} satisfies $\alpha\Lf \in [1, 3]$, 
	comparing the upper bounds in Proposition~\ref{prop.JNesterovLikeBounds} with the lower bounds in Theorem~\ref{thm.LowerJ} shows that, for settling times $\Ts\le(\kappa+1)/2$,~\eqref{eq.parametersNesterovLike} achieves order-wise optimal $J_{\max}$ and $J_{\min}$ for both the iterate $(\sigma_w=\sigma)$ and gradient $(\sigma_w=\alpha\sigma)$ noise models.
	
	\begin{figure}[t]
		\centering
		\begin{tabular}{c}
			\begin{tabular}{c}
				\resizebox{6.5 cm}{!}{
					\begin{tikzpicture}
						\node[] (pic) at (0,0) {\includegraphics[width=70mm]{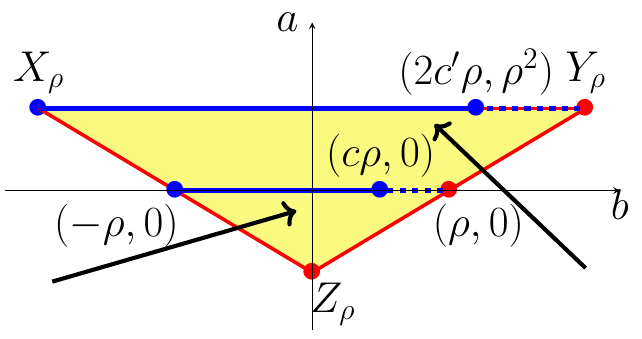}};
						\node[name=gdparam] at (-3,-1.7)
						{
							\color{black}
							\large Parameters~\eqref{eq.gradientDescentReducedAlpha}
						};
						
						\node[name=hbparam] at (3,-1.7)
						{
							\color{black}
							\large Parameters~\eqref{eq.parametersRHB}
						};
					\end{tikzpicture}
				}
			\end{tabular}
		\end{tabular}
		\vspace*{-0.25cm}
		\caption{The triangle $\Delta_\rho$ (yellow) and the line segments $(b(\lambda),a(\lambda))$ with $\lambda\in[\mf,\Lf]$ (blue) for gradient descent with reduced stepsize~\eqref{eq.gradientDescentReducedAlpha} and heavy-ball-like method~\eqref{eq.parametersRHB}, which place the end point $(b(\mf),a(\mf))$ at  $X_\rho$ and the end point $(b(\Lf),a(\Lf))$ at $(2c'\rho,\rho^2)$ on the edge $X_\rho Y_\rho$, where $c'\DefinedAs\kappa(1-\rho)^2/\rho -(1+\rho^2)/\rho$ ranges over the interval $[-1,1]$.}
		\vspace*{-0.5cm}
		\label{fig.TunningReducedAlpha-structure}
	\end{figure}
	
	\vspace*{-2ex}
	\subsection{Impact of reducing the stepsize}
	
	When the only source of uncertainty is a noisy gradient, i.e., $\sigma_w=\alpha\sigma$, one can attempt to reduce the noise amplification $J$ by decreasing the stepsize $\alpha$ at the expense of increasing the settling time $\Ts=1/(1-\rho)$~\cite{yuayinsay16,fazribmor18,aybfalgurozd19b,vanles21}. In particular, for gradient descent, $\alpha$ can be reduced from its optimal value $2/(\Lf+\mf)$ by keeping $(b(\mf),a(\mf))$ at $(-\rho, 0)$ and moving the point $(b(\Lf),a(\Lf))$ from $(\rho,0)$ towards $(-\rho, 0)$ along the horizontal axis; see Figure~\ref{fig.TunningReducedAlpha-structure}. This can be accomplished with 
	\begin{align}\label{eq.gradientDescentReducedAlpha}
		\alpha\;=\; (1\,+\,c\rho)/\Lf,\quad \gamma\;=\;\beta\;=\;0 
	\end{align} 
	for some $c\in[-1,1]$ parameterizing $(b(\Lf),a(\Lf))=(c\rho,0)$. In this case, the settling time satisfies $\Ts=(\kappa+c)/(c+1)\in[(\kappa+1/2),\infty)$ and similar arguments to those presented in the proof of Lemma~\ref{lem.individualJhatHeavyBallLike} can be used to obtain
	\begin{align*}
		\hat{J}_{\max}
		&\;=\;
		\hat{J}(\mf)
		\,=\,
		\sigma^2 \kappa^2(1-\rho)/{\Lf^2}
	\end{align*}
	\begin{align*}
		\hat{J}_{\min}
		&\,=\,
		\left\{
		\ba{rcll}
		\hat{J}(\Lf)
		&\!\!\!=\!\!\!&
		{\sigma^2 \alpha^2}/({1-c^2\rho^2})
		& c\;\le\; 0
		\\[0.1cm]
		\hat{J}(1/\alpha)
		&\!\!\!=\!\!\!&
		\sigma^2\alpha^2
		& c\;\ge\; 0.
		\ea
		\right.
	\end{align*}
	For a fixed $n$, the stepsize in~\eqref{eq.gradientDescentReducedAlpha} yields a $\Theta(\kappa^2)$ scaling for both $J_{\max}\times\Ts$ and  $J_{\min}\times\Ts$ for all $c\in[-1,1]$. Thus, gradient descent with reduced stepsize order-wise matches the lower bounds in Theorem~\ref{thm.LowerJ}. An IQC-based approach~\cite[Lemma~1]{mohrazjovTAC21} was utilized in~\cite[Theorem~13]{vanles21} to show that $\alpha$ in~\eqref{eq.gradientDescentReducedAlpha} also yields the above discussed convergence rate and worst-case noise amplification for one-point $m$-strongly convex $\Lf$-smooth functions.
	
	\begin{myrem}
		Any desired settling time $\Ts=1/(1-\rho)\in[(\sqrt{\kappa}+1)/2,\infty)$ can be achieved by the heavy-ball-like method with reduced stepsize,
		\begin{align}\label{eq.parametersRHB}
			\alpha\;=\;(1-\rho)^2/\mf, \quad \beta\;=\;\rho^2,\quad \gamma\;=\;0.
		\end{align} 
		This choice yields $J_{\max}=\sigma^2 n\kappa^2(1-\rho^4)/(L^2 (1+\rho)^4)$ for the gradient noise model $\sigma_w=\alpha\sigma$~\cite[Theorem~9]{vanles21}; see  Figure~\ref{fig.TunningReducedAlpha-structure}. In addition, by considering the error in $y^t=x^t+\gamma(x^t-x^{t-1})$ as the performance metric, it was stated and numerically verified in~\cite{vanles21} that the choice of parameters~\eqref{eq.parametersRHB} yields Pareto-optimal algorithms for simultaneously optimizing $J_{\max}$ and $\rho$ for the gradient noise model $\sigma_w=\alpha\sigma$. We note that the settling time $\Ts=\Theta(\kappa)$ of gradient descent with standard stepsizes ($\alpha=1/\Lf$ or $2/(\mf+\Lf)$) can be achieved via~\eqref{eq.parametersRHB} by reducing $\alpha$ to $O(1/(\kappa \Lf))$. In contrast, the parameterized family of heavy-ball-like methods~\eqref{eq.parametersHeavy-ballLike} is order-wise Pareto-optimal (cf.~Theorems~\ref{thm.LowerJ}-\ref{thm.LowerBoundSettlingTimeBased}) while maintaining $\alpha \in [1/\Lf,4/\Lf]$.
	\end{myrem}

\begin{figure}[t]
	\centering
	\begin{tabular}{c}
		\begin{tabular}{c}
			\resizebox{5.0 cm}{!}{
				\begin{tikzpicture}
					\node[] (pic) at (0,0) {\includegraphics[width=70mm]{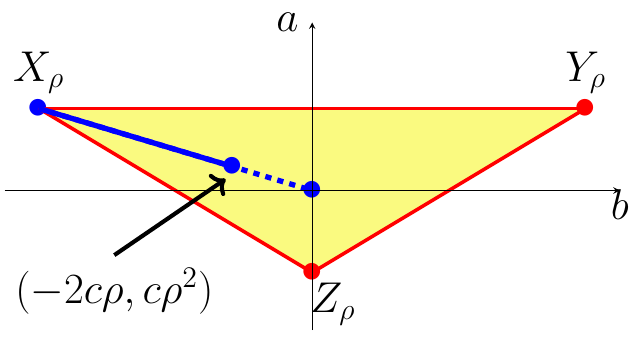}};
				\end{tikzpicture}
			}
		\end{tabular}
	\end{tabular}
	\vspace*{-0.25cm}
	\caption{The triangle $\Delta_\rho$ (yellow) and the line segments $(b(\lambda),a(\lambda))$ with $\lambda\in[\mf,\Lf]$ (blue) for the Nesterov-like method with reduced stepsize~\eqref{eq.gradientDescentReducedAlpha}, which places the end point $(b(\mf),a(\mf))$ at  $X_\rho$ and the end point $(b(\Lf),a(\Lf))$ at $(-2c\rho,c\rho^2)$, where $c\DefinedAs\kappa-1/(1-\rho)^2$ ranges over the interval $[0,1]$.}
	\vspace*{-0.5cm}
	\label{fig.TunningMultiStageLocus}
\end{figure}

\begin{figure}[t]
	\centering
	\begin{tabular}{c}
		\resizebox{4.2 cm}{!}{
			\begin{tikzpicture}
				\node[] (pic) at (0,0) {\includegraphics[]{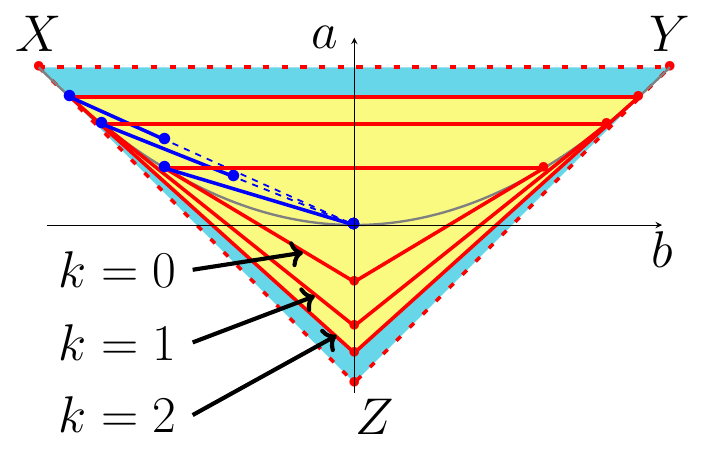}};
			\end{tikzpicture}
		}
	\end{tabular}
	\vspace*{-0.25cm}
	\caption{The exponential stability triangles $\Delta_{\rho_k}$ along with the line segments $(b_k(\lambda),a_k(\lambda))$ with $\lambda\in[\mf,\Lf]$ associated with  the first three stages $k=0,1,2$ of M-ASG. The parabola $b^2=4a$ determines the locus of vertices $X_\rho$ and $Y_\rho$.}
	\vspace*{-0.5cm}
	\label{fig.TunningMASGStagesLocus}
\end{figure}

\begin{myrem}{
		Any desired settling time $\Ts=1/(1-\rho)\in[\sqrt{\kappa},\infty)$ can be achieved by the Nesterov-like method with reduced stepsize,
		\begin{align}\label{eq.parametersRNA}
			\alpha\;=\;(1-\rho)^2/\mf, \quad \beta\;=\;\gamma\;=\;\dfrac{\rho}{2\,-\,\rho}.
		\end{align} 
		This choice makes the line segment $(b(\lambda),a(\lambda))$ for $\lambda\in[\mf,\Lf]$ pass through the origin with the endpoint $(b(\mf),a(\mf))$ at the vertex $X_\rho=(-2\rho,\rho^2)$ of the $\rho$-exponential stability triangle $\Delta_\rho$, and it yields  $(b(\Lf),a(\Lf))=(-2c\rho,c\rho^2)$, where $c\DefinedAs\kappa-T_s^2$ is a non-negative constant; see Figure~\ref{fig.TunningMultiStageLocus}. 
		For the gradient noise model $\sigma_w = \alpha\sigma$, these parameters lead to
		\begin{align}\label{eq.JhatmaxRNA}
			\hat{J}_{\max}
			\;=\;
			\hat{J}(\mf)
			\;=\;
			\dfrac{\sigma^2 \kappa^2(1\,+\,\rho^2)}{L^2(1\,+\,\rho)^3 T_s}
			\;\approxeq\;
			\dfrac{\sigma^2 \kappa^2 }{4 \Lf^2 T_s} 
		\end{align}
		We note that, as $\rho\rightarrow1$, the largest modal contribution to noise amplification $\hat{J}_{\max}$ becomes inversely proportional to the settling time $T_s$. The family of parameters in~\eqref{eq.parametersRNA} were utilized in~\cite{aybfalgurozd19b} to propose the Multistage Accelerated Stochastic Gradient (M-ASG) algorithm as a means to systematically tradeoff convergence rate and noise amplification. For strongly convex problems, this algorithm optimally reduces the error in function values, thereby matching the fundamental lower bound established in~\cite{nemyud83}. In particular, at every stage $k\in\{0,1,\ldots\}$, \mbox{M-ASG} performs a specific restart that balances the initial condition followed by $N_k$ Nesterov-like-iterations with 
		\begin{align}\label{eq.multistage-parameters}
			\alpha_k \;\DefinedAs\;\dfrac{1}{4^k \Lf},~~ \beta_k\;=\;\gamma_k\;\DefinedAs\;\dfrac{1 \, - \, \sqrt{\alpha_k m}}{1 \, + \,\sqrt{\alpha_k m}}
		\end{align} 
		where $N_k$ is proportional to $2^k/\kappa$, $\mf$ is the parameter of strong convexity, and $\Lf$ is the Lipschitz constant. When restricted to strongly convex quadratic problems, the parameters in~\eqref{eq.multistage-parameters} are identical to those in~\eqref{eq.parametersRNA} with the convergence rate
		$
		\rho_k
		=
		1-{1}/{(2^k\sqrt{\kappa})}
		$; see Figure~\ref{fig.TunningMASGStagesLocus}.
	It is straightforward to show that while M-ASG reduces the largest contribution to noise amplification $\hat{J}(\mf)$ to half by going to the next stage, it also doubles the settling time. 
		Finally, contrasting~\eqref{eq.JhatmaxRNA} with the lower bounds in~\eqref{eq.lowerboundJhatMGradientNoise} and~\eqref{eq.thmLowerJMax_withAlpha}, established in Theorem~\ref{thm.LowerBoundSettlingTimeBased}, allows us to conclude that M-ASG preserves $J\times T_s$ near the Pareto-optimal curve at each stage while achieving the optimal iteration complexity~\cite{nemyud83} by successively reducing the stepsize to half of its previous value and utilizing a suitable iteration count~$N_k$.		
		}
\end{myrem}

		\vspace*{-2ex}
	\section{Proofs of Theorems~\ref{thm.mainUpperbound}-\ref{thm.mainResultHBLike}}
	\label{sec.mainResultProof}
	
	\subsection{Proof of Theorem~\ref{thm.mainUpperbound}}
	
	From Theorem~\ref{thm.varianceJhat} it follows that we can use upper bounds on $\hat{J}(\lambda)$ over $\lambda\in[\mf,\Lf]$ to establish an upper bound on $J$. Since the algorithm achieves the convergence rate $\rho$, combining equation~\eqref{eq.rateDecompos} and Lemma~\ref{lem.expstability} yield
	$
	(b(\lambda),a(\lambda))
	\in
	\Delta_\rho
	$
	for all $\lambda\in[\mf,\Lf]$.  
	As we demonstrate in Appendix~\ref{app.convexity}, the function $\hat{J}$ is convex in $(b,a)$ over the stability triangle $\Delta$. In addition, $\Delta_\rho\subset \Delta$ is the convex hull of the points $X_\rho$, $Y_\rho$, $Z_\rho$ in the ($b,a$)-plane. Since the maximum of a convex function over the convex hull of a finite set of points is attained at one of these points, 
	$\hat{J}$ attains its maximum over $\Delta_\rho$ at $X_\rho$, $Y_\rho$, or $Z_\rho$. 
	
	Using the definition of $X_\rho$, $Y_\rho$, and $Z_\rho$ in~\eqref{eq.VerticesDeltaRho}, the affine relations~\eqref{eq.dhlDefinition}, and the analytical expression for $\hat{J}$ in Theorem~\ref{thm.varianceJhat}, it follows that the maximum occurs at vertices $X_\rho$ and $Y_\rho$,
	\[
	\hat{J}_{\max}
	\;\DefinedAs\;
	\max_{\lambda\in[\mf,\Lf]} \hat{J}(\lambda)
	\;=\;
	\sigma_w^2(1+\rho^2)/\left((1-\rho)^3(1+\rho)^3\right)
	\]
	where we use
	$
	d_{X_\rho}=l_{Y_\rho}=(1-\rho)^2
	$, 
	$
	l_{X_\rho}=d_{Y_\rho}=(1+\rho)^2
	$, 
	and
	$
	h_{X_\rho}=h_{Y_\rho}=1-\rho^2.
	$
	Combining the above identity with Theorem~\ref{thm.varianceJhat} completes the proof of~\eqref{eq.mainUpperBoundsigma1}. 
	
	We use an argument similar to the proof of Proposition~\ref{prop.recoverNesLowerbound} to prove~\eqref{eq.mainUpperBoundsigmaalpha}. In particular, since $(b(\Lf),a(\Lf))\in\Delta_\rho$, we have
\begin{align*}
	\alpha L
	\;=\;
	d(\Lf)
	\;\le\;
	d_{\max}
	\;=\;
	(1 \, + \, \rho)^2
\end{align*}
where $d$ given by~\eqref{eq.dDef} is the horizontal signed distance to the edge $XZ$ of the stability triangle $\Delta$. On the other hand, $d_{\max}$ is the largest value that $d$ can take among all points $(b,a)\in\Delta_\rho$ and it corresponds to the vertex $Y_\rho$; see equation~\eqref{eq.Dmax}. 
Combining this inequality with $\sigma_w=\alpha\sigma$ and~\eqref{eq.mainUpperBoundsigma1} completes the proof of Theorem~\ref{thm.mainUpperbound}.

		\vspace*{-2ex}
	\subsection{Proof of Theorem~\ref{thm.LowerJ}}
	Using the expression $J=\sum_i \hat{J}(\lambda_i)$ established in Theorem~\ref{thm.varianceJhat}, we have  the decomposition
	\begin{align}\label{eq.decompositionProof}
		J
		&\;=\;
		\hat{J}(\mf) \, + \, \sum_{i \, = \, 1}^{n-1}\hat{J}(\lambda_i).
	\end{align} 
	To prove the lower bounds~\eqref{eq.thmLowerJMin} and~\eqref{eq.thmLowerJMin_withAlpha} on $J_{\min}$, we establish a lower bound on $\hat{J}(\mf)\times\Ts$ that scales quadratically with $\kappa$, and a general lower bound on $\hat{J}(\lambda)\times \Ts$. 
	
	\subsubsection*{Case $\sigma_w=\sigma$}
	
	The proof of~\eqref{eq.thmLowerJMin} utilizes the inequalities  
	\begin{subequations}\label{eq.lowerboundJhat}
		\begin{align}\label{eq.lowerboundJhatM}
			\hat{J}(\mf)\times\Ts
			&\; \ge \;
			\sigma_w^2\kappa^2/(2(1 \, + \, \rho)^5)
			\\[0.1cm]\label{eq.lowerboundJhatGeneral}
			\hat{J}(\lambda)\times\Ts
			&\; \ge \;
			\sigma_w^2(\sqrt{\kappa}\,+\,1)/2.
		\end{align}
	\end{subequations}
	
	We first prove~\eqref{eq.lowerboundJhatM}. Our approach builds on the proof of Proposition~\ref{prop.recoverNesLowerbound}. In particular, $d(\lambda)=\alpha\lambda$ for the point $(b(\lambda),a(\lambda))$, where $d$ and $(b,a)$ are  defined in~\eqref{eq.dhlDefinition} and~\eqref{eq.ab}, respectively. Thus,
	$
	d(\mf)
	=
	d(\Lf)/\kappa.
	$
	Furthermore, Lemma~\ref{lem.expstabilityAlgorithm} implies $(b(\lambda),a(\lambda))\in\Delta_\rho$ for $\lambda\in[\mf,\Lf]$. Thus, the trivial inequality $d(\Lf)\le d_{\max}$ leads to
	\begin{align}\label{eq.lowerBoundD}
		d(\mf)
		\;\le\;
		{d_{\max}}/{\kappa}
		\;=\;
		{(1 \, + \, \rho)^2}/{\kappa}
	\end{align}
	where $d_{\max}=(1+\rho)^2$ is the largest value that $d$ can take among all points $(b,a)\in\Delta_\rho$; see equation~\eqref{eq.Dmax}. 
	We now use Theorem~\ref{thm.varianceJhat}  to write
	\begin{align}\label{eq.lowerBoundNoiseProof}
		\hat{J}(\lambda)
		&\;=\;
		\dfrac{\sigma_w^2(d(\lambda)\,+\,l(\lambda))}{2 d(\lambda) h(\lambda) l(\lambda)}
		\;\ge\;
		\dfrac{\sigma_w^2}{2 d(\lambda) h(\lambda)}.
	\end{align}
Next, we lower bound the right-hand side of~\eqref{eq.lowerBoundNoiseProof}. Let $\mathcal{L}$ be the line that passes through $(b(\lambda),a(\lambda))$ which is parallel to the edge $XZ$ of the stability triangle $\Delta$, and let $G$ be the intersection of $\mathcal{L}$ and the edge $X_\rho Z_\rho$ of the $\rho$-stability triangle $\Delta_\rho$; see Figure~\ref{fig.proofLowerbound} for an illustration. It is easy to verify that
	\begin{subequations}\label{eq.geomArgHelper1}
		\begin{align}\label{eq.geomArgHelper1Sub1}
			h_G
			\;\ge\;
			h(\lambda),\quad 
			d_G
			\;=\;
			d(\lambda)
		\end{align}
		where $h_G$ and $d_G$ correspond to the values of $h$ and $d$ associated with the point $G$.
		In addition, since $G$ lies on the edge $X_\rho Z_\rho$, $h_G$ and $d_G$ satisfy the affine relation
		\begin{align}\label{eq.geomArgHelper1Sub2}
			h_G
			\;=\;
			1\,-\,\rho
			\,+\,
			d_G \rho /({1\,-\,\rho}).
	\end{align}\end{subequations}
	This follows from the equation of the line $X_\rho Z_\rho$ in the ($b,a$)-plane and from the definitions of $d$ and $h$ in~\eqref{eq.dhlDefinition}.
Furthermore, combining~\eqref{eq.geomArgHelper1Sub1} and~\eqref{eq.geomArgHelper1Sub2} implies
	\begin{subequations}\label{eq.cTempFull}
\begin{align}\label{eq.cTemp}
	\hspace{-0.1 cm}
	h(\lambda)(1\,-\,\rho)
	\;\le\;
	h_G(1\,-\,\rho)
	\;=\;
	(1\,-\,\rho)^2
	\,+\,
	\rho\, d(\lambda).
	\hspace{-0.1 cm}
\end{align}
		For $\lambda=\mf$, we can further write
	\begin{align}\label{eq.cTempTemp}
	2\,d(\mf) \left((1\,-\,\rho)^2
		\,+\,
		\rho\, d(\mf)\right) 
	\;\le\;
	2\,(1+\rho)^5/\kappa^2
\end{align}
	\end{subequations}
	where the inequality is obtained from~\eqref{eq.boudnKappaGeneral} and~\eqref{eq.lowerBoundD}. Combining~\eqref{eq.lowerBoundNoiseProof},~\eqref{eq.cTemp}, and~\eqref{eq.cTempTemp} completes the proof of~\eqref{eq.lowerboundJhatM}.
		
	\begin{figure}[h]
		\begin{center}
			\begin{tabular}{c}
				\resizebox{5.3 cm}{!}{
					\begin{tikzpicture}
						\node[] (pic) at (0,0) {\includegraphics[]{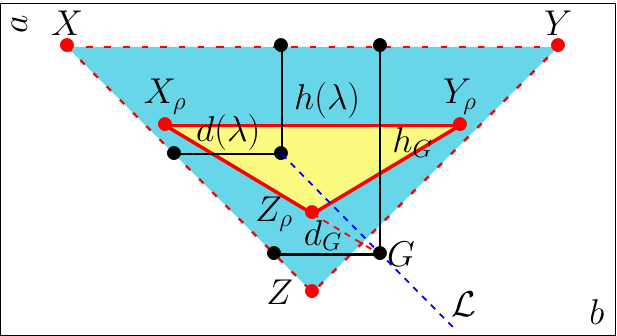}};
					\end{tikzpicture}
				}
			\end{tabular}
		\end{center}
				\vspace*{-0.3cm}
		\caption{The line $\mathcal{L}$ (blue, dashed) and the intersection point $G$, along with the distances $d_1$, $h_1$, $d_G$, and $h_G$ as introduced in the proof of Theorem~\ref{thm.LowerJ}.}
		\vspace*{-0.25cm}
		\label{fig.proofLowerbound}
	\end{figure}
	
	Next, we prove the general lower bound in~\eqref{eq.lowerboundJhatGeneral}.
	As we demonstrate in Appendix~\ref{app.convexity}, the modal contribution $\hat{J}$ to the noise amplification is a convex function of $(b,a)$ which takes its minimum $\hat{J}_{\min}=\sigma_w^2$ over the stability triangle $\Delta$ at the origin $b=a=0$. Combining this fact with the lower bound in~\eqref{eq.NesFunLow} on $\rho$ completes the proof of~\eqref{eq.lowerboundJhatGeneral}.
	
	Finally, we can obtain the lower bound~\eqref{eq.thmLowerJMin} on $J_{\min}$ by combining~\eqref{eq.decompositionProof} and~\eqref{eq.lowerboundJhat}.

	\subsubsection*{Case $\sigma_w=\alpha\sigma$}
	The proof of~\eqref{eq.thmLowerJMin_withAlpha} utilizes the inequalities  
	\begin{subequations}\label{eq.lowerboundJhatGradientNoise}
		\begin{align}\label{eq.lowerboundJhatMGradientNoise}
			\hat{J}(\lambda)\times\Ts
			&\;\ge\;
			\sigma^2/(2 \lambda^2(1\,+\,\rho))
			\\[0.cm]\label{eq.lowerboundJhatGeneralGradientNoise}
			\hat{J}(\lambda)\times\Ts
			&\; \ge \;
			{\sigma^2(1\,-\,\rho)^3\kappa^2}/{L^2}.
		\end{align}
	\end{subequations}
In particular,~\eqref{eq.thmLowerJMin_withAlpha} follows from using~\eqref{eq.lowerboundJhatMGradientNoise} for $\lambda=\mf$ and taking the maximum of~\eqref{eq.lowerboundJhatMGradientNoise} and~\eqref{eq.lowerboundJhatGeneralGradientNoise} for the other eigenvalues to bound the expression for $J$ in Theorem~\ref{thm.varianceJhat}.
	
	We first prove~\eqref{eq.lowerboundJhatMGradientNoise}. By combining~\eqref{eq.lowerBoundNoiseProof} and~\eqref{eq.cTemp}, we obtain
	\begin{align}\label{eq.lowerboundProofTemp3}
		\dfrac{\hat{J}(\lambda)}{1\,-\,\rho}
		&
		\;\ge\;
		\dfrac{\alpha^2\sigma^2}{2\,d(\lambda) \left((1\,-\,\rho)^2
			\,+\,
			\rho\, d(\lambda)\right) }.
	\end{align}
	Since
	$
	d(\lambda)
	\ge
	d_{\min}
	\DefinedAs
	(1 -  \rho)^2,
	$
	where $d_{\min}$ is the smallest value of $d$ over $\Delta_\rho$, [cf.~\eqref{eq.Dmin}], we can write
	\begin{align}\label{eq.dTemp}
		\dfrac{\alpha^2\sigma^2}{ (1\,-\,\rho)^2
			\,+\,
			\rho d(\lambda) }
		\;\ge\;
		\dfrac{\alpha^2\sigma^2}{ d(\lambda)(1\,+\,\rho) }
		\;=\;
		\dfrac{\sigma^2 d(\lambda)}{ \lambda^2(1\,+\,\rho)}.
	\end{align} 
	Combining~\eqref{eq.lowerboundProofTemp3} and~\eqref{eq.dTemp} completes the proof of~\eqref{eq.lowerboundJhatMGradientNoise}.
	
	To prove~\eqref{eq.lowerboundJhatGeneralGradientNoise} we use 
	$
	d(\lambda)
	\ge
	d_{\min}
	\DefinedAs
	(1 -  \rho)^2
	$
	and $d(\mf)=\alpha\mf$, to obtain
	$
	\alpha 
	\ge
	(1-\rho)^2\kappa/\Lf. 
	$
	Combining this inequality with~$\hat{J}_{\min}=\sigma_w^2=\alpha^2\sigma^2$ yields~\eqref{eq.lowerboundJhatGeneralGradientNoise}. Finally, we obtain the lower bound~\eqref{eq.thmLowerJMin_withAlpha} on $J_{\min}$ by combining~\eqref{eq.decompositionProof} and~\eqref{eq.lowerboundJhatGradientNoise}.

	To obtain the lower bounds~\eqref{eq.thmLowerJMax} and~\eqref{eq.thmLowerJMax_withAlpha} on $J_{\max}$, we  consider a quadratic function for which the Hessian has $n-1$ eigenvalues at $\lambda=\mf$ and  one eigenvalue at $\lambda=\Lf$. For such a function, we can use Theorem~\ref{thm.varianceJhat} to write
	\begin{align}\label{eq.decompositionProofWorstCase}
		J_{\max}
		&\;\ge\;
		J
		\;=\;
		(n\,-\,1)\hat{J}(\mf) + \hat{J}(\Lf).
	\end{align}
	\subsubsection*{Case $\sigma_w=\sigma$}
	To prove~\eqref{eq.thmLowerJMax}, we use inequalities in~\eqref{eq.lowerboundJhatM} and~\eqref{eq.lowerboundJhatGeneral} to bound $\hat{J}(\mf)/(1-\rho)$ and $\hat{J}(\Lf)/(1-\rho)$ in~\eqref{eq.decompositionProofWorstCase}, respectively.
	
	\subsubsection*{Case $\sigma_w=\alpha\sigma$}
	To prove~\eqref{eq.thmLowerJMax_withAlpha}, we use inequality in~\eqref{eq.lowerboundJhatMGradientNoise} with $\lambda=\mf$ to lower bound $\hat{J}(\mf)/(1-\rho)$, and combine~\eqref{eq.lowerboundJhatMGradientNoise} and~\eqref{eq.lowerboundJhatGeneralGradientNoise} to lower bound $\hat{J}(\Lf)/(1-\rho)$ in~\eqref{eq.decompositionProofWorstCase}.

	\vspace*{-2ex}
	\subsection{Proof of Theorem~\ref{thm.LowerBoundSettlingTimeBased}}
	\label{sec.Thm4proof}
	
	The following proposition allows us to prove the lower bounds in Theorem~\ref{thm.LowerBoundSettlingTimeBased}.
	\vsp
	\begin{myprop}\label{prop.LowerBoundSettlingTimeBased}
		Let $\rho=\spec(A)=1-1/\Ts$ be the convergence rate of the two-step momentum algorithm~\eqref{alg.TM_disc}. Then, the largest and smallest modal contributions to noise amplification given by~\eqref{eq.JhatMinMaxDetermineJMinMax} satisfy
$
	\hat{J}_{\max}
\ge
	\sigma_w^2 /(2(1\,+\,\rho)^2) \, \Ts
$
and
$
	\hat{J}_{\min}
	\ge
	\sigma_w^2.
$
	\end{myprop}
	\vsp
\begin{proof}
	The inequality $\hat{J}_{\min} \ge \sigma_w^2$ follows from the fact that $\hat{J}$, as a function of  $(b,a)$, takes its minimum value  at the origin; see Appendix~\ref{app.convexity}.   
	The proof for $\hat{J}_{\max}$ utilizes the fact that for any constant parameters $(\alpha,\beta,\gamma)$	and fixed condition number, the spectral radius $\spec(A)$ corresponds to the smallest $\rho$-linear convergence triangle $\Delta_\rho$ that contains the line segment $(b(\lambda),a(\lambda))$ for $\lambda\in[\mf,\Lf]$. Thus, at least one of the end points $(b(\mf),a(\mf))$ or $(b(\Lf),a(\Lf))$ will be on the boundary of the triangle $\Delta_{\spec(A)}$. Combining this with the fact that $d(\mf)\le d(\Lf)$, it follows that at least one of the following holds
	\begin{align*}
		(b(\mf),a(\mf)) 
		&\;\in\;
		X_\rho Z_\rho \text{ or } X_\rho Y_\rho,
		\\
		(b(\Lf),a(\Lf)) 
		&\;\in\;
		Y_\rho Z_\rho \text{ or } X_\rho Y_\rho.
	\end{align*}
	Together  with the concrete values of vertices~\eqref{eq.VerticesDelta} in terms of $\rho$, this yields
	\begin{align}\label{eq.heperProofKappaIndepLowerBoundJhat1}
		1\,-\,\rho
		\;\ge\;
		\min \left\{
		h(\mf),\,h(\Lf),\, l(\Lf)/(1\,+\,\rho),\,d(\mf)/(1\,+\,\rho)
		\right\}
	\end{align}
	Also, using Theorem~\ref{thm.varianceJhat} and noting that the maximum values that $h(\lambda)$, $d(\lambda)$, and $l(\lambda)$ can take among $\Delta_\rho$ are given by
	$1+\rho^2$,
	$(1+\rho)^2$, and
	$(1+\rho)^2$,
	respectively, we can write
	\begin{align}\label{eq.heperProofKappaIndepLowerBoundJhat2}
		\begin{split}
			\hat{J}(\mf)
			&\;\ge\;
			\dfrac{\sigma_w^2}{2h(\mf) d(\mf)}
			\;\ge\;
			\max \left\{
			\dfrac{\sigma_w^2}{2h(\mf) (1+\rho)^2},
			\dfrac{\sigma_w^2}{2 d(\mf) (1+\rho^2)}
			\right\},
			\\[0.15cm]
			\hat{J}(\Lf)
			&\;\ge\;
			\dfrac{\sigma_w^2}{2h(\Lf) l(\Lf)}
			\;\ge\;
			\max \left\{
			\dfrac{\sigma_w^2}{2h(\Lf) (1+\rho)^2}
			,
			\dfrac{\sigma_w^2}{2 l(\Lf) (1+\rho^2)}
			\right\}.
		\end{split}
	\end{align}
	Finally, by the convexity of $\hat{J}$ (see Appendix~\ref{app.convexity}), we have $\hat{J}_{\max}\ge\max\{ \hat{J}(\mf),\hat{J}(\Lf)\}$. Combining this   with~\eqref{eq.heperProofKappaIndepLowerBoundJhat1} and~\eqref{eq.heperProofKappaIndepLowerBoundJhat2} completes the proof. 
\end{proof}	
	
	\vsp
	The lower bounds in Theorem~\ref{thm.LowerBoundSettlingTimeBased} follow from combining Proposition~\ref{prop.LowerBoundSettlingTimeBased} with the expression for $J$ in Theorem~\ref{thm.varianceJhat}.

	\vspace*{-2ex}
	\subsection{Proof of Theorem~\ref{thm.mainResultHBLike}}
	As described in Section~\ref{sec.Tuning}, the parameters in Theorem~\ref{thm.mainResultHBLike} are obtained by placing the end points of the horizontal line segment $(b(\lambda),a(\lambda))$ parameterized by $\lambda\in[\mf,\Lf]$ at the edges $X_\rho Z_\rho$ and $Y_\rho Z_\rho$ of the $\rho$-linear convergence triangle $\Delta_\rho$. These parameters can be equivalently represented  by~\eqref{eq.parametersHeavy-ballLike} where the scalar $c$ given in Lemma~\ref{lem.individualJhatHeavyBallLike} satisfies $c\in[0,1]$ if and only if $\Ts\le(\kappa+1)/2$ and it satisfies $c\in[-1,0]$ if and only if $\Ts\ge (\kappa+1)/2$. The proof of Theorem~\ref{thm.mainResultHBLike} follows from combining Lemma~\ref{lem.individualJhatHeavyBallLike} and Propositions~\ref{prop.JHeavyBallLikeBounds} and~\ref{prop.HBlikeNegative}.

	\vspace*{-1ex}
	\section{Concluding remarks}
	\label{sec.concludingRemarks}
	
	We have examined the amplification of stochastic disturbances for a class of two-step momentum algorithms in which the iterates are perturbed by an additive white noise which arises from uncertainties in gradient evaluation or in computing the iterates. For both noise models, we establish lower bounds on the product of the settling time and the smallest/largest steady-state variance of the error in the optimization variable. These bounds scale with $\kappa^2$ for all stabilizing parameters, which reveals a fundamental limitation imposed by the condition number $\kappa$ in designing algorithms that tradeoff noise amplification and convergence rate. In addition, we provide a novel geometric viewpoint of stability and $\rho$-linear convergence. This viewpoint brings insight into the relation between noise amplification, convergence rate, and algorithmic parameters. It also allows us to (i) take an alternative approach to optimizing convergence rates for standard algorithms; (ii) identify key similarities and differences between the iterate and gradient noise models; and (iii) introduce parameterized families of algorithms for which the parameters can be continuously adjusted to tradeoff noise amplification and settling time. By utilizing positive and negative momentum parameters in accelerated and decelerated regimes, respectively, we demonstrate that a parameterized family of the heavy-ball-like algorithms can achieve order-wise Pareto optimality for all settling times and both noise models. 
		
	Our ongoing work focuses on extending these results to algorithms with more complex structures including update strategies that utilize information from more than the last two iterates and time-varying algorithmic parameters~\cite{padsei22}. It is also of interest to identify fundamental performance limitations of stochastic gradient descent algorithms in which both additive and multiplicative stochastic disturbances exist~\cite{huseiran17,huseiles21}. 
	
		\vspace*{-1ex}
	\section*{Acknowledgments}
	
	We thank Laurent Lessard for his comments on an earlier draft of this manuscript. 
	
		\vspace*{-1ex}
	\appendix
	\subsection{Settling time}\label{app.settlingTime}
	
	If $\rho$ denotes the linear convergence rate, $\Ts=1/(1-\rho)$ quantifies the {\it settling time}. The inequality in~\eqref{eq.linearConvergenceFormal} shows that $c\rho^t\le \epsilon$ provides a sufficient condition for reaching the accuracy level $\epsilon$ with $\norm{\psi^t}_2/\norm{\psi^0}_2\le\epsilon$. Taking the logarithm of $c\rho^t\le \epsilon$ and using the first-order Taylor series approximation $\log \, (1-x)\approx -x$ around $x = 0$ yields a sufficient condition on the number of iterations $t$ for an algorithm to reach $\epsilon$-accuracy,
	\[
	t
	\,\ge\,
	\log \, (\epsilon/c)/\log \, (1 - 1/\Ts)
	\;\approx\;
	\Ts\log \, (c/\epsilon).
	\]
	In continuous time, the sufficient condition for reaching $\epsilon$-accuracy $c \mathrm{e}^{-\rho t}\le \epsilon$ yields
	$
	t
	\ge
	\log \, (c/\epsilon)/{\rho},
	$
	and $\Ts=1/\rho$ can be used to asses the settling time.
	
		\vspace*{-2ex}
	\subsection{Convexity of modal contribution $\hat{J}$ to noise amplification}
	\label{app.convexity}
	To show the convexity of $\hat{J}$, we use the fact that the function $g(x)=\prod_{i=1}^{d}x_i^{-1}$ is convex over the positive orthant $\R^d_{++}$. This can be verified by noting that its Hessian satisfies
	\[
	\nabla^2g(x)
	\;=\;
	g(x)\left(\diag(x) \,+\, xx^T\right)\succ 0
	\] 
	where $\diag(\cdot)$ is the diagonal matrix.
	By Theorem~\ref{thm.varianceJhat}, we have
	\begin{align*}
		\dfrac{\hat{J}}{\sigma_w^2}
		&\;=\;
		\dfrac{d\,+\,l}{2\,d\,h\,l}
		\;=\;
		\dfrac{1}{2\,h\,d}
		\,+\,
		\dfrac{1}{2\,h\,l}
	\end{align*}
	where we have dropped the dependence on $\lambda$ for simplicity. The functions $1/(2hd)$ and $1/(2hl)$ are both convex over the positive orthant $d,h,l>0$. Thus, $\hat{J}$ is convex with respect to $(d,h,l)$. In addition, since $d,h$, and $l$ are all affine functions of $a$ and $b$, we can use the equivalence relation in~\eqref{eq.equivalenceDelta} to conclude that $\hat{J}$ is also convex in $(b,a)$ over the stability triangle $\Delta$. Finally, since $b(\lambda)$ and $a(\lambda)$ are affine in $\lambda$, it follows that for any stabilizing parameters, $\hat{J}$ is also convex with respect to $\lambda$ over the interval $[\mf,\Lf]$.
	
	Convexity of $\hat{J}$ allows us to use first-order conditions to find its minimizer. In particular, since for $\sigma_w = 1$
	\begin{align*}
		\dfrac{\partial \hat{J}}{\partial d}
		&\;=\;
		-\dfrac{1}{2hd^2},\quad
		\dfrac{\partial \hat{J}}{\partial l}
		\;=\;
		-\dfrac{1}{2hl^2},\quad
		\dfrac{\partial \hat{J}}{\partial h}
		\;=\;
		-\dfrac{l\,+\,d}{2h^2dl}
		\\[0.15 cm]
		\dfrac{\partial d}{\partial a}
		&\;=\;
		\dfrac{\partial l}{\partial a}
		\;=\;
		-\dfrac{\partial h}{\partial a}
		\;=\;
		\dfrac{\partial d}{\partial b}
		\;=\;
		-\dfrac{\partial l}{\partial b}
		\;=\;
		1,
		\quad
		\dfrac{\partial h}{\partial b}
		\;=\;
		0
	\end{align*}
	it is easy to verify that 
	$
	{\partial \hat{J}}/{\partial a}
	=
	{\partial \hat{J}}/{\partial b}
	=
	0
	$
	at $a=b=0$. Thus,  $\hat{J}$ takes its minimum  
	$\hat{J}_{\min}=\sigma_w^2$  over the stability triangle $\Delta$ at $a=b=0$, which corresponds to $d=h=l=1$.
		\vspace*{-2ex}
	\subsection{Proofs of Section~\ref{sec.modalDecomp}}
	\label{app.ModalDecomp}
	
	\subsubsection{Proof of Lemma~\ref{lem.expstability}}
	We start by noting that
	$
	\spec(M)
	\le\rho
	$ if and only if
	$\spec(M')\le1
	$
	where $M' \DefinedAs M/\rho$.
	The characteristic polynomial associated with $M'$, 
	$
	F_\rho (z) = z^2 + ({b}/{\rho})z + {a}/{\rho^2},
	$
	allows us to use similar arguments to those presented in the proof of Lemma~\ref{lem.stability} to show that
	\begin{align}\label{eq.proofExpStabHelper1}
		\spec(M')\,\le\, 1
		\quad
		\iff
		\quad 
		(b/\rho,\,a/\rho^2)
		\,\in\,
		\Delta_1
	\end{align}
	where $
	\Delta_1
	\DefinedAs	
	\left\{ 
	(b,a) 
	\; | \;
	|b| \, - \, 1 \, \leq \, a \, \leq \, 1
	\right\}
	$
	is the closure of the set $\Delta$ in~\eqref{eq.Delta}. Finally, the condition on the right-hand side of~\eqref{eq.proofExpStabHelper1} is equivalent to
	$(b,a)\in\Delta_\rho$, where $\Delta_\rho$ is given by~\eqref{eq.Deltarho}.  
	
		\vspace*{-2ex}
	\subsection{Proofs of Section~\ref{sec.Tuning}}
	\label{app.Tuning}
	\subsubsection{Proof of Lemma~\ref{lem.individualJhatHeavyBallLike}}
	
	We show that the parameters ($\alpha,\beta,\gamma$) in~\eqref{eq.parametersHeavy-ballLike} place the points $(b(\mf),a(\mf))$ and $(b(\Lf),a(\Lf))$ on the edges $X_\rho Z_\rho$ and $Y_\rho Z_\rho$ of the $\rho$-linear convergence triangle $\Delta_\rho$, respectively. In particular, we can use a scalar $c\in[-1,1]$ to parameterize the end points as
	\be
	\ba{rcl}
	(b(\mf),a(\mf))
	&\!\!\!=\!\!\!&
	(-(1+c)\rho,\;c\rho^2)
	\\[0.15 cm]
	(b(\Lf),a(\Lf))
	&\!\!\!=\!\!\!&
	((1+c)\rho,\;c\rho^2).
	\ea
	\non
	\ee
	Using the definition of $a$ and $b$ in~\eqref{eq.ab}, we can solve the above equations for ($\alpha,\beta,\gamma$) to verify the desired parameters. Thus, the algorithm achieves the convergence rate $\rho$. In addition  the points $c=0$ and $c=1$ recover gradient descent and heavy-ball method with the parameters that optimize the convergence rate; see Table~\ref{tab:rates}.
	
	Furthermore, $h$, $d$, and $l$ in~\eqref{eq.dhlDefinition} are given by 
	\begin{subequations}
		\be
		\ba{rclrl}
		h(\mf)
		&\!\!\!=\!\!\!&
		h(\Lf)
		&\!\!\!=\!\!\!&
		1
		\,-\,
		c\rho^2
		\\[0.1cm]
		d(\mf)
		&\!\!\!=\!\!\!&
		l(\Lf)
		&\!\!\!=\!\!\!&
		(1\,-\,\rho)(1\,-\,c\rho)
		\\[0.1cm]
		l(\mf)
		&\!\!\!=\!\!\!&
		d(\Lf)
		&\!\!\!=\!\!\!&
		(1\,+\,\rho)(1\,+\,c\rho)
		\ea
		\label{eq.dhlHeayBallLike}
		\ee
		and the condition number is determined by
		\begin{align}\label{eq.kappaHeavyBallLikeProof}
			\kappa
			\;=\;
			\alpha \Lf/(\alpha \mf)
			\;=\;
			{d(\Lf)}/{d(\mf)}
			\;=\;
			{l(\mf)}/{d(\mf)}.
		\end{align}
		Combining~\eqref{eq.kappaHeavyBallLikeProof} with~\eqref{eq.dhlHeayBallLike}, and rearranging terms yields the desired expression for $c$ in terms of $\rho$ and $\kappa$.
		
		The analytical expressions in Theorem~\ref{thm.varianceJhat} imply that for the parameters in~\eqref{eq.parametersHeavy-ballLike}, the function $\hat{J}(\lambda)$ is symmetric over $[\mf,\Lf]$, i.e.,
		$
		\hat{J}(\lambda)
		=
		\hat{J}(\mf + \Lf - \lambda)
		$ for all $\lambda\in[\mf,\Lf]$. In addition, as we demonstrate in Appendix~\ref{app.convexity}, $\hat{J}(\lambda)$ is convex. Thus, $\hat{J}(\lambda)$ attains its maximum at $\lambda=\mf$ and $\lambda=\Lf$ and we can use the expression for $\hat{J}(\lambda)$ in Theorem~\ref{thm.varianceJhat} to obtain the maximum,
		\begin{align}\label{eq.proofHeavyBallLikeHelper1}
			\hat{J}(\mf)
			\;=\;
			\dfrac{\sigma_w^2(d(\mf) + l(\mf))}{2 h(\mf) d(\mf) l(\mf) }
			\;=\;
			\dfrac{\sigma_w^2(\kappa \,+\, 1)}{2 h(\mf) l(\mf) }
		\end{align}
		where the second equality follows from~\eqref{eq.kappaHeavyBallLikeProof}. Combining~\eqref{eq.dhlHeayBallLike} and~\eqref{eq.proofHeavyBallLikeHelper1} yields the expression for $\hat{J}(\mf)$.
		
		Also, symmetry and convexity imply that $\hat{J}(\lambda)$ attains its minimum at the midpoint 
		$\lambda=\hat{\lambda}\DefinedAs(\mf+\Lf)/2 
		= 
		(1+\beta)/\alpha
		$. 
		This point corresponds to $(b(\hat{\lambda}),a(\hat{\lambda}))= (0,c\rho^2)$ in the ($b,a$)-plane and it thus satisfies
		\begin{align}\label{eq.dhlForLambdaBar}
			h(\hat{\lambda})\,=\, 1\,-\,c\rho^2, 
			\;\;
			d(\hat{\lambda})\,=\, l(\hat{\lambda})\,=\, 1 \,+\, c\rho^2.
		\end{align}
	\end{subequations}
	Using~\eqref{eq.dhlForLambdaBar} to evaluate the expression for $\hat{J}(\lambda)$ at the point $\lambda=\hat{\lambda}$ yields the desired minimum value.
	
	\subsubsection{Proof of Proposition~\ref{prop.JHeavyBallLikeBounds}}
	Using the expressions established in Lemma~\ref{lem.individualJhatHeavyBallLike}, it is straightforward to verify that
	\begin{align*}
		\hat{J}(\mf)\times \Ts
		&\;=\;
		\sigma_w^2 \, p_{1c} (\rho) \kappa(\kappa \,+\, 1)
		\\[0.cm]
		\hat{J}(\hat{\lambda})
		\times\Ts
		&\;=\;
		\sigma_w^2 \,\kappa \, p_{2c} (\rho)
	\end{align*}
	and that, for the gradient noise model ($\sigma_w=\alpha\sigma$), we have 
	\begin{align*}
		\hat{J}(\mf)\times \Ts
		&\;=\;
		\sigma^2 \, p_{3c} (\rho) \kappa(\kappa \,+\, 1)
		\\[0.cm]
		\hat{J}(\hat{\lambda})
		&\;=\;
		\sigma^2 \,\kappa \, p_{2c} (\rho)
	\end{align*}
	where the functions $p_{1c}(\rho)$-$p_{4c}(\rho)$ are given by~\eqref{eq.p_iDef}. Thus, the expressions for $J_{\max}$ and $J_{\min}$ follow from Corollary~\ref{cor.JHeavyBallLike}. The bounds on $p_{1c}(\rho)$-$p_{4c}(\rho)$ follow from the fact  that, for $\rho \in (0,1)$, we have
	\[
	q_c (\rho)
	\;=\;
	\dfrac{1 - c \rho}{1 - c \rho^2}
	\;\in\;
	\left\{
	\ba{ll}
	\,[1/(1+c\rho), 1] & \quad c \, \in \, [0,1]
	\\
	\,[1/2,2]
	& \quad c \, \in \, [-1,0] .
	\ea
	\right.
	\]
	This completes the proof.
	
	\subsubsection{Proof of Proposition~\ref{prop.HBlikeNegative}}
	
	Using the expressions established in Lemma~\ref{lem.individualJhatHeavyBallLike}, it is straightforward to verify that
	\begin{align*}
		\hat{J}(\mf)
		&\;=\;
		\sigma_w^2 \, p_{5c}(\rho)\,(1 \,+\, 1/\kappa)\,\Ts
		\\[0.cm]
		\hat{J}(\hat{\lambda})
		&\;=\;
		\sigma_w^2 p_{6c}(\rho) \Ts/\kappa
	\end{align*}
	where $p_{5c}$ and $p_{6c}$ are given by Proposition~\ref{prop.HBlikeNegative}. Thus, the expressions for $J_{\max}$ and $J_{\min}$ follow from Corollary~\ref{cor.JHeavyBallLike}. The bounds on $p_{5c}$ and $p_{6c}$ follow from $c \in [-1,0]$ and $\rho \in (0,1)$.
	
	\vspace*{-2ex}
	\subsection{Lyapunov equation and the steady-state variance}
	\label{app.LyapuEquation}
	
	For discrete-time LTI system~\eqref{eq.ss}, the covariance matrix 
	$
	P^t
	\DefinedAs  
	\EX 
	\left(
	\psi^t
	(\psi^t)^T
	\right)
	$
	of the state vector $\psi^t$ satisfies the linear recursion
	\begin{subequations}
		\be
		P^{t+1} 
		\; = \; 
		A \, P^t A^T 
		\; + \; 
		B B^T
		\label{eq.LyapPt}
		\ee
		and its steady-state limit 
		\be
		P
		\; \DefinedAs \; 
		\lim_{t \, \to \, \infty} 
		\,
		\EX 
		\left[
		\psi^t
		(\psi^t)^T
		\right]
		\label{eq.P}
		\ee 
		is the unique solution to the algebraic Lyapunov equation~\cite{kwasiv72},
		\be
		P
		\; = \; 
		A \, P A^T 
		\; + \; 
		B B^T.
		\label{eq.Lyap}
		\ee
		For stable LTI systems, performance measure~\eqref{eq.Jnew} can be computed using 
		\be
		J
		\; = \;
		\lim_{t \, \to \, \infty} 
		\, 
		\dfrac{1}{t}
		\sum_{k \, = \, 0}^{t}
		\trace
		\left( Z^k \right)
		\; = \;
		\trace
		\left( Z \right)
		\label{eq.Jus-P}
		\ee
	\end{subequations} 	
	where 
	$
	Z = C P C^T	
	$	
	is the steady-state limit of the output covariance matrix 
	$
	Z^t
	\DefinedAs  
	\EX 
	\left[
	z^t
	(z^t)^T
	\right]
	=
	C P^t C^T.
	$
	We can prove Theorem~\ref{thm.varianceJhat} by finding the solution $P$ to~\eqref{eq.Lyap} for the two-step momentum algorithm. 
	

\end{document}